\documentclass{amsart}

\usepackage{amsmath,amssymb,amsthm, mathrsfs, mathtools, bm, amsfonts}
\usepackage{mathabx}
\usepackage{amscd,mathtools}
\usepackage[utf8]{inputenc}
\usepackage[english]{babel}
\usepackage{cite}
\usepackage{color}
\usepackage[pagebackref,colorlinks,citecolor=blue,linkcolor=red]{hyperref}

\usepackage{float}
\usepackage{tikz}
\usepackage{lmodern}
\usetikzlibrary{decorations.pathmorphing}

\newcommand{\PSL}{{\rm{PSL}}}

\usepackage{multicol}
\tikzset{snake it/.style={decorate, decoration=snake}}
\usetikzlibrary{automata}
\usetikzlibrary{cd}

\usepackage{comment}
\usepackage{mathtools}

\usepackage{csquotes}

\usepackage{amsthm}

\usepackage{cleveref}

\newtheorem{theorem}{Theorem}[section]
  \newtheorem{proposition}[theorem]{Proposition}

            \newtheorem{remark}[theorem]{Remark}
    \newtheorem{defn}{Definition}

  \newtheorem{corollary}[theorem]{Corollary}
  \newtheorem{lemma}[theorem]{Lemma}

  \newtheorem{question}[theorem]{Question}

\newcommand{\rank}{{\rm{rank}}}
\newcommand{\sys}{{\rm{sys}}}
\newcommand{\combsys}{{\rm combsys}}
\newcommand{\Tet}{{\rm {\frak Tet}}}
\newcommand{\CAT}{{\rm CAT}}
\newcommand{\sss}{\hspace{-.05cm} \smallsetminus \hspace{-.05cm}}

\definecolor{Green}{RGB}{0,180,60}
\definecolor{Purple}{RGB}{150,0,250}

\title{Filling links and essential systole}

\author{Christopher J Leininger}\thanks{CJL was supported by NSF grant DMS-2305286 and DMS–1928930}
\address{University of Notre Dame}
\email{cleining@nd.edu}
\urladdr{\url{https://sites.google.com/view/chris-leiningers-webpage/home}}

\author{Yandi Wu}
\address{Rice University}
\email{yw220@rice.edu}
\urladdr{\url{https://yandiwu.github.io/}}

\begin{document}

\maketitle

\begin{abstract}
    We answer a question of Freedman and Krushkal, producing {\em filling links} in any closed, orientable $3$--manifold.  The links we construct are hyperbolic, and have large {\em essential systole}, contrasting earlier geometric constraints on hyperbolic links in $3$--manifolds due to Adams--Reid and Lakeland--Leininger.
\end{abstract}

\section{Introduction}

Recall that the \textit{systole} of a manifold $M$, denoted sys$(M)$, is the length of its shortest closed geodesic. A link $L \subset M$ in a closed 3-manifold is {\em hyperbolic} if $M \sss L$ admits a complete, finite--volume hyperbolic structure (which is unique by Mostow--Prasad Rigidity \cite{mostow,prasad}). This paper is partly motivated by the following theorem due to Adams and Reid, which states that there exists a uniform upper bound on the systole of the complement of {\bf any} hyperbolic link in a fixed closed, non-hyperbolic $3$--manifold.

\begin{theorem}[Adams--Reid, Theorem 1 in \cite{adamsreid}] \label{thm:adamsreid} Let $M$ be a closed, orientable 3-manifold admitting no hyperbolic metric. Let $L \subset M$ be a hyperbolic link. Then \\ $\text{sys}(M \sss L) \leq 7.35534\cdots$.
    
\end{theorem}

When $M$ is hyperbolic, a theorem due to Lakeland and the first author \cite{lakelandleininger} similarly produces an upper bound on the systole of the complement of any hyperbolic link which depends only on the hyperbolic volume of $M$. In particular, for any closed, orientable 3-manifold $M$, there is an upper bound on sys$(M \sss L)$ for {\bf any} hyperbolic link $L \subset M$.  On the other hand, the short geodesics produced in $M \sss L$ in both \cite{adamsreid} and \cite{lakelandleininger} are nullhomotopic in the ambient manifold $M$.  One is naturally led to define the {\em essential systole} of a link $L \subset M$ of a closed $3$--manifold $M$ to be
\[ {\sys}_e(L) = \inf\{\ell(\gamma) \mid \gamma \mbox{ is a loop in } M \sss L \mbox{ with } i \circ \gamma \mbox{ non-null-homotopic in } M \}, \]
and ask: is there is also a uniform bound for the essential systole for $L \subset M$, independent of $L$?
Contrasting \cite{adamsreid} and \cite{lakelandleininger}, we prove that the answer is `no'.

\newcommand{\UnbddEss}{Given a closed, orientable $3$--manifold $M$ with $\pi_1(M) \neq \{1\}$ and $r > 0$, there exists a hyperbolic link $L \subset M$ such that
\[ {\sys}_e(L) > r.\]}

\begin{theorem} \label{thm:unbounded essential systole}
    \UnbddEss
\end{theorem}





We note that the assumption that $M$ is not simply connected is clearly necessary since otherwise all loops in $M \sss L$ are null-homotopic, and ${\sys}_e(L)$ is trivially bounded.

\begin{remark}
    One could alternatively define the essential systole of $L$ to be the length of the shortest {\em geodesic} in $M \sss L$ which is non-null-homotopic in $M$ (instead of the shortest loop).  Of course, Theorem~\ref{thm:unbounded essential systole} also holds with this definitions as well.
    As a consequence of our definition, however, observe that if ${\sys}_e(L) > 0$, then every parabolic element of $\pi_1(M \sss L)$ is necessarily in the kernel of the induced map $i_* \colon \pi_1(M \sss L) \to \pi_1(M)$ from inclusion, since such an element is represented by loops with arbitrarily small length.
\end{remark}



An interesting application of \Cref{thm:unbounded essential systole} revolves around a theme in $3$-manifold topology: the existence of links and knots that ``go everywhere" inside their ambient $3$-manifolds. The notion of ``going everywhere" has been made precise in a number of ways. For example, in work preceding the resolution of the Poincar\'e Conjecture, Bing showed that any $3$-manifold $M$ which is not diffeomorphic to $S^3$ contains a knot which cannot be contained in any $3$-balls in $M$ (see \cite{bing}). Myers defined \textit{disk-busting curves} as knots that intersect every essential sphere or disk and showed that every $3$-manifold contains a disk-busting curve (see \cite{myers}). 

As a further example relevant to the present paper, Freedman and Krushkal formulate the notion of a link which cuts through the entire topology of (or ``fills") its ambient 3-manifold in \cite{FreeKrus2023}. Roughly speaking, they define a \textit{filling link} as one that kills the relations in \textit{any} group presentation of $\pi_1(M)$ with a minimal generating set. More precisely, given a closed $3$--manifold, $M$, a {\em $1$--spine} for $M$ (also called a {\em carrier graph}) is a graph $\Gamma$ with $\rank(\pi_1(\Gamma)) = \rank(\pi_1(M))$ and map $f \colon \Gamma \to M$ such that $f_* \colon \pi_1(\Gamma) \to \pi_1(M)$ is surjective.
Given a link $L \subset M$, let $i \colon M \sss L \to M$ be the inclusion.  We say that $f \colon \Gamma \to M \sss L$ is an $L$--relative $1$--spine if $i \circ f \colon \Gamma \to M$ is a $1$--spine.
Freedman and Krushkal declare that a link $L \subset M$ is a {\em filling link} if for every $L$--relative $1$--spine $f \colon \Gamma \to M \sss L$, we have
\[ f_*\colon \pi_1(\Gamma) \to \pi_1(M \sss L) \]
is injective. They then ask about the existence of such filling links. 
\begin{question}[Freedman-Krushkal \cite{FreeKrus2023}] \label{Q:FK} Does every closed $3$--manifold contain a filling link?
\end{question}
\Cref{Q:FK} is trivial if $\pi_1(M) \cong \{1\}$; indeed, for any link $L \subset M$, a point is the only $L$--relative $1$--spine, and so $L$ is filling.  If $\pi_1(M)$ is cyclic, any hyperbolic knot $K \subset M$ (see \cite{myers}) is filling since its fundamental group is torsion free, and a $K$--relative $1$--spine is just a loop which is non-null-homotopic in $M$. \color{black}
In the appendix to \cite{FreeKrus2023}, the first author and Reid answered \Cref{Q:FK} affirmatively when rank$(\pi_1(M)) = 2$. When rank$(\pi_1(M)) = 3$, an affirmative answer is due to Stagner \cite{stagner}. The techniques of Leininger-Reid and Stagner are specific to ranks 2 and 3 respectively and do not generalize to other ranks. We use \Cref{thm:unbounded essential systole} to answer the question affirmatively for an arbitrary closed, orientable $3$--manifold.

\begin{theorem} \label{thm:main}
    Every closed, orientable $3$--manifold $M$ contains a hyperbolic filling link.
\end{theorem}

The idea is to find a hyperbolic link $L$ so that for any $L$--relative $1$--spine $f \colon \Gamma \to M \sss L$ and any basis for $\pi_1(\Gamma)$, the $f_*$--image of the basis elements have large translation length; hyperbolic geometry then forces $f_*$ to be injective. The intuition for this construction comes from the work of White \cite{white} and independently Kapovich and Weidmann \cite{kapovichweidmann} (see below). White proved that the systole of a closed hyperbolic $3$--manifold $M$ is bounded by a function of the rank of its fundamental group.  His proof involves analyzing $1$--spines $f \colon \Gamma \to M$ of minimal length, and then showing that if the systole is sufficiently large, then $f_*$ must be injective, contradicting the fact it is a $1$--spine of a closed hyperbolic $3$--manifold.  For more on probing the geometry of hyperbolic $3$--manifolds via their $1$--spines, see \cite{BachCoopWhit,Biringer,BiringerSouto}, for example.

White's theorem can be seen as a consequence of a theorem of Kapovich and Weidmann \cite{kapovichweidmann}, which appeared at roughly the same time as White's paper \cite{white}.  We appeal to this result of Kapovich and Weidmann (see Theorem~\ref{thm:KW}), and easily deduce the following.

\newcommand{\BigEssFill}{Given $n \geq 1$, there exists $R>0$ so that if $M$ is a closed $3$ manifold with $\rank(\pi_1(M)) = n$, and $L \subset M$ is a hyperbolic link with ${\sys}_e(L) > R$, then $L$ is a filling link.}

\begin{theorem} \label{thm:big essential systole filling}
\BigEssFill    
\end{theorem}

Theorem~\ref{thm:main} is now an easy consequence of \Cref{thm:big essential systole filling} and \Cref{thm:unbounded essential systole}.

\begin{proof}[Proof of Theorem~\ref{thm:main} assuming Theorems~\ref{thm:big essential systole filling} and \ref{thm:unbounded essential systole}.]
    Suppose $\rank(\pi_1(M)) = n$ and let $R > 0$ be as in Theorem~\ref{thm:big essential systole filling}.  By Theorem~\ref{thm:unbounded essential systole}, there exists a link $L\subset M$ so that ${\sys}_e(L) > R$, which is thus filling by Theorem~\ref{thm:big essential systole filling}.
\end{proof}

\bigskip

\noindent
{\bf Outline of the paper and sketch of the proof.} In Section~\ref{S:hyperbolic geometry and large essential systole} we recall some basics about Nielsen equivalence, \S\ref{S:Nielsen}, and hyperbolic geometry, \S\ref{S:hyperbolic},  then conclude by proving Theorem~\ref{thm:big essential systole filling} in \S\ref{S:reduction theorem}. Section~\ref{S:links big esssys} contains the construction of links with large essential systole.  The first step is to construct a triangulation of our arbitrary closed manifold $M$ with controlled local properties, but large {\em combinatorial systole}, following a construction due to Cooper and Thurston \cite{CoopThur1988}; see \S\ref{Sec:Triangulations}.  Next, in \S\ref{Sec:Tet tangle}, we explicitly construct a tangle in a tetrahedron whose complement admits a nice hyperbolic structure with totally geodesic boundary and dihedral angles equal to $\frac{\pi}2$.  This tangle serves as the building block for the link with large essential systole that we construct in \S\ref{S:final construction}.  Specifically, starting with the triangulation $\tau$ of a closed $3$--manifold $M$ constructed in \S\ref{Sec:Triangulations} having large combinatorial systole, we delete the tangle from each tetrahedron.  The resulting link $L_\tau$ has an explicit, complete $\CAT(-1)$ metric, and hence admits a complete hyperbolic metric by Thurston's Hyperbolization Theorem \cite{Thurston-hyp}.  
While we do not have explicit control over the hyperbolic metric, we use the $\CAT(-1)$ metric to analyze a family of surfaces with bounded Euler characteristic which serve as ``barriers" to accessing the tetrahedra.  We show that having large combinatorial systole implies any loop $\gamma$ in $M \sss L_\tau$ which is non-null-homotopic in $M$ must intersect many tetrahedra, and hence many of the surfaces.  Every intersection of $\gamma$ with one of the surfaces passes through a point with bounded injectivity radius, and separation properties of the surfaces imply that $\gamma$ must pass through many such distinct points, and is therefore long.

\begin{remark}
    Because Kapovich-Weidmann work in the setting of arbitrary $\delta$--hyperbolic spaces, we could avoid using Thurston's Hyperbolization Theorem to prove that the links we construct are filling, and shorten the argument a little bit; see \Cref{sec:conclusion}. 
\end{remark}

\bigskip

\noindent
{\bf Acknowledgements.}  The authors would like to thank Alan Reid for helpful conversations, including the reference to \cite{Saratch} mentioned at the end of the paper.  The authors would also like to thank Slava Krushkal and Ian Biringer for comments on an earlier version of the paper, including Biringer's suggestions that led to Theorems~\ref{thm:fullrank1} and \ref{thm:fullrank2}.  The first author is particularly grateful for his collaboration with Reid on the appendix of \cite{FreeKrus2023} and the discussions during that time regarding filling links that have influenced this paper. Finally, the authors thank the referee for carefully reading the paper, and making several helpful suggestions that greatly clarified the exposition.

\section{Hyperbolic geometry and large essential systole} \label{S:hyperbolic geometry and large essential systole}

The goal of this section is to prove the following.

\medskip

\noindent
{\bf Theorem~\ref{thm:big essential systole filling}} {\em \BigEssFill}

\medskip

\subsection{Nielsen equivalence} \label{S:Nielsen} 

Nielsen transformations were originally introduced by Nielsen in \cite{nielsen} to prove that every subgroup of a finitely generated free group is free. Since then, they have been used extensively to study properties of free groups (see \cite{frs}). In this section, we recall basic notions from \cite{nielsen}. 

\begin{defn}[Nielsen Equivalence] Let $G$ be a group, and consider an $n$--tuple of elements, $(g_1, g_2, g_3, ..., g_n) \in G^n$. Define the following \emph{elementary Nielsen moves} on $(g_1,...,g_n)$: 

\begin{enumerate}
    \item For some $1 \leq i \leq n$, replace $g_i$ with $g_i^{-1}$ in $(g_1,...,g_n)$; 
    \item For $i \neq j$, $1 \leq i, j \leq n$, replace $g_i$ with $g_ig_j$ in $(g_1,...,g_n)$; 
    \item For $i \neq j$, $1 \leq i, j \leq n$, interchange $g_i$ and $g_j$ in $(g_1,...,g_n)$. 
\end{enumerate}

A \emph{Nielsen transformation} is a finite sequence of elementary Nielsen moves. If $(g_1,...,g_n)$ and $(g'_1,...,g'_n) \in G^n$ differ by a Nielsen transformation, we say they are \emph{Nielsen equivalent}, denoted $(g_1,...,g_n) \sim_{N} (g'_1,...,g'_n)$. 
\end{defn}

Recall that the \textit{basis} of a finite rank free group is a minimal collection of generators of $\mathbb{F}_n$. In \cite{nielsen}, Nielsen showed that if $G$ is a finite rank free group, then the set of Nielsen transformations on a basis generates Aut$(G)$. The following lemma is immediate.

\begin{lemma}\label{L:nielsenbasis}
Let $\phi \in \text{Hom}(\mathbb{F}_n, G)$, where $\mathbb{F}_n$ is a finite rank free group. Given a basis $\{b_1, b_2, \ldots , b_n\}$ of $\mathbb{F}_n$, let $c_i = \phi(b_i)$ ($1 \leq i \leq n$), and let $L_1 = (c_1, c_2, \ldots , c_n) \in G^n$. Let $L_2 = (c'_1, c'_2, \ldots , c'_n) \in G^n$ such that $L_1 \sim_{N} L_2$. Then there exists $\{b'_1, b'_2, ..., b'_n\}$, a basis for $\mathbb{F}_n$, such that $\phi(b'_i) = c'_i$ for all $1 \leq i \leq n$. 
\end{lemma}

\begin{proof} 
Note that $c_i = \phi(b_i)$ implies that $c_i^{-1} = \phi(b_i^{-1})$ and $c_ic_j = \phi(b_ib_j)$. Thus, if $L_1 = (c_1,c_2,\ldots,c_n)$ and $L_2 = (c_1',c_2',\ldots,c_n')$ differs from $L_1$ by a Nielsen transformation, then we can lift the Nielsen transformation to $(b_1,\ldots,b_n)$ to produce $(b_1',b_2',\ldots,b_n')$ so that $\phi(b_i') = c_i'$ for all $i$.  The resulting tuple $(b_1',b_2',\ldots,b_n')$ differs from $(b_1, b_2,\ldots, b_n)$ by a Nielsen transformation, which is an element of Aut$(\mathbb F_n)$. Therefore $(b_1',b_2',\ldots,b_n')$ is also a basis of $\mathbb{F}_n$. 
\end{proof}

\subsection{Hyperbolic space}\label{S:hyperbolic}

In this section, we recall basic facts about hyperbolic space which will be useful in the proof of \Cref{thm:big essential systole filling}.  See, for example, Chapter 2 of \cite{Thurston-Book}.

We consider the upper-half space model of \textit{hyperbolic 3-space},
\[ \mathbb{H}^3 = \left\{(x_1, x_2, x_3) \in \mathbb{R}^3: x_n > 0 \right\}\]
endowed with the Riemannian metric $\frac{dx_1^2 + dx_2^2 + dx_3^2}{x_3^2}$.  We let $\rho$ denote the associated distance function. The group of orientation-preserving isometries of $(\mathbb{H}^3, \rho)$ is PSL$_2(\mathbb{C})$, the set of $2 \times 2$ matrices with determinant $1$ and complex entries, modulo $\{\pm I\}$. The action is by conformal extension of linear fractional transformations on the $(x_1,x_2)$--plane, viewed as $\mathbb C$. A \textit{hyperbolic 3-manifold} is the quotient of $\mathbb{H}^3$ by a discrete, torsion-free subgroup of PSL$_2(\mathbb{C})$. 

There are two ``coarser" forms of hyperbolicity which are also important to this paper. First, an $n$-manifold $X$ is \emph{locally CAT($-1$)} if every point has a neighborhood in which triangles are at least as thin as those in $\mathbb H^2$; by the Cartan-Hadamard Theorem \cite[Theorem~II.4.1]{BridHaef1999}, its universal cover $\tilde{X}$ is CAT($-1$) ({\em all} triangles satisfy the above comparison). The following consequence of Thurston's Hyperbolization Theorem allows us to endow the locally CAT($-1$) link complements we will construct in this paper with hyperbolic structures. Furthermore, the locally CAT($-1$) and hyperbolic manifolds share the same coarse geometry in their universal covers. 

\begin{theorem} \label{T:CAT-1 to hyp}
Suppose a 3--manifold $N$ admits a complete, finite-volume, locally CAT(-1), piecewise hyperbolic metric which is hyperbolic in the cusps. Then $N$ admits a complete, finite-volume hyperbolic metric.  Moreover, this metric can be chosen so that the identity on the universal covers is a quasi-isometry.
\end{theorem}

The piecewise hyperbolic structure is what we produce in our construction, and it allows us to easily make sense of ``finite volume".  In fact, for the proof we could simply assume that $N$ is complete, locally CAT(-1), and has finitely many ends, each isometric to a finite volume hyperbolic cusp.

\begin{proof} The universal cover $\widetilde N$ of $N$, being a  complete, CAT(-1) $3$--manifold, is noncompact and contractible; see e.g.~\cite[Corollary II.5.1]{BridHaef1999}.  Consequently, $N$ is irreducible and $\pi_1N$ is infinite.  By the Flat Torus Theorem \cite[Theorem~II.7.1]{BridHaef1999}, $N$ is atoroidal.  By the Hyperbolization Theorem, $N$ admits a complete, finite volume hyperbolic metric; see \cite{thurston82,Thurston-hyp,Morgan,McMullen}.

The hyperbolic and locally CAT(-1) metrics both have hyperbolic cusps.  These are isometric to quotients of horoballs in $\mathbb H^3$, and the metrics are warped products on a flat surface crossed with $[0,\infty)$ (warped by the scalar function $e^{-t}$, for $t \in [0,\infty)$).  In particular, after an isotopy we can assume the two metrics differ by a fixed affine deformation in the first factor, and by the identity in the second factor.  Consequently, the identity from $N$ to itself is biLipschitz on the cusps.  On the remaining compact complement of the cusps, we can homotope the identity maps to Lipschitz maps which are homotopy inverses of each other.  Lifting these Lipschitz maps to the universal covers gives Lipschitz maps from the CAT(-1) metric to the hyperbolic metric and vice versa, which are all bounded distance from the identity. These maps are therefore quasi-isometries; thus so is the identity.
\end{proof}

Second, recall that a geodesic metric space, $X$, is \textit{$\delta$-hyperbolic} if for any $x, y, z \in X$, the geodesic segment $[x, y]$ lies in the $\delta$-neighborhood of the union of geodesic segments $[x, z] \cup [y, z]$. In other words, all triangles in $X$ are \textit{$\delta$-thin}. A metric space is \emph{Gromov hyperbolic} if it is $\delta$-hyperbolic for some $\delta \geq 0$. 

Since Gromov hyperbolic spaces are modeled on hyperbolic space, $\mathbb{H}^3$ is a classic example of one. To see this, it suffices to consider a triangle in the upper half plane $\mathbb{H}^2$ obtained by restricting $\rho$ to the subspace with $x_1 =0$.
Triangles in $\mathbb{H}^2$ have area $\pi - \alpha - \beta - \gamma$, where $\alpha$, $\beta$, and $\gamma$ are the interior angles of the triangle. The maximum area of a triangle in $\mathbb H^2$ is $\pi$, which is realized by an ideal triangle with vertices at $-1$, $1$, and $\infty$, for example. Note that $\delta$ is the distance between $i$ and the geodesic $x_2 = 1$. This distance is bounded above by the distance between $i$ and $1 + i$, which is $\text{ln}\left(\frac{3 + \sqrt{5}}{2}\right) < 1$. So we can set $\delta = 1$ to be a $\delta$-hyperbolicity constant for $\mathbb{H}^3$. 

\subsection{Large essential systole implies filling} \label{S:reduction theorem}

We will deduce Theorem~\ref{thm:big essential systole filling} from a general theorem of Kapovich-Weidmann from \cite{kapovichweidmann} about actions on Gromov hyperbolic spaces (whose origins they attribute to Gromov).  We will only need to apply their result in the case that the Gromov hyperbolic space is $\mathbb H^3$, and will not need the full strength of their conclusion.  We therefore state only the form we will need, but emphasize that their result is more general and has a stronger conclusion.
\begin{theorem}[\!\cite{kapovichweidmann}] \label{thm:KW} For any integer $n > 0$, there exist a constant $C(n)$ with the following property. Suppose a group
\[ G =  \langle g_1,\ldots,g_n\rangle < \PSL_2(\mathbb C).\]
Then one of the following holds:
\begin{enumerate}
    \item The group $G$ is free with basis $(g_1,\ldots,g_n)$, or
    \item There exists an $n$-tuple $(g_1',\ldots,g_n')$ and $y \in \mathbb{H}^3$ such that $\rho(g_1'\cdot y,y) < C(n)$ and $(g_1,\ldots,g_n) \sim_{N} (g_1',\ldots,g_n')$.
\end{enumerate}
\end{theorem}

In other words, if $G$ acts by isometries on $\mathbb{H}^3$, then $G$ is either free or contains a nontrivial element with small translation length. We are now ready to prove \Cref{thm:big essential systole filling}.

\begin{proof}[Proof of Theorem~\ref{thm:big essential systole filling}] Suppose $M$ is a closed $3$--manifold with $\rank(\pi_1(M)) = n$, and suppose $L \subset M$ is a hyperbolic link with ${\sys}_e(L) \geq C(n)$, from Theorem~\ref{thm:KW}.  Now let 
\[ f :\Gamma \to M \sss L \]
be an $L$-relative $1$-spine.  Choose any basis $(h_1,\ldots,h_n)$ for $\pi_1(\Gamma)$, and we write $g_j = f_*(h_j)$, for each $j$.  By Theorem~\ref{thm:KW}
\[ \langle g_1,\ldots,g_n \rangle = f_*(\pi_1(\Gamma))\]
is free with basis $(g_1,\ldots,g_n)$, or $(g_1,\ldots,g_n)$ is Nielsen equivalent to $(g_1',\ldots,g_n')$ so that for some $y \in \mathbb H^3$, we have $\rho(g_1' \cdot y,y) < C(n)$.  
First, we suppose we are in the latter case and derive a contradiction.

Let $i \colon M \sss L \to M$ be the inclusion, and note that $i_*(g_1')$ is necessarily trivial in $\pi_1(M)$ since it represents a loop with length less than ${\sys}_e(L)$.
By \Cref{L:nielsenbasis}, the Nielsen moves on $(g_1,\ldots,g_n)$ lift to Nielsen moves on $(h_1,\ldots,h_n)$, producing a new basis $(h_1',\ldots,h_n')$ of $\pi_1(\Gamma)$ so that $f_*(h_j') = g_j'$.  Since 
\[ (i  \circ f)_*(h_1') = i_*(f_*(h'_1)) = i_*(g_1'),\] 
we see that
\[ \pi_1(M) = i_*(f_*(\pi_1(\Gamma))) 
   = \langle i_*(g_1'),\ldots,i_*(g_n'))\rangle
   = \langle i_*(g_2'),\ldots,i_*(g_n'))\rangle,\]
contradicting the fact that $\rank(\pi_1(M)) = n$.

Therefore, $f_*(\pi_1(\Gamma))$ is free on $(g_1,\ldots,g_n)$, and hence $f_*$ is injective by the Hopfian property of free groups.  Since $f \colon \Gamma \to M \sss L$ was an arbitrary $L$--relative $1$--spine, it follows that $L$ is a filling link. 
\end{proof}

Theorem~\ref{thm:big essential systole filling} could also be proved using White's arguments in \cite{white}.  Specifically, White considers $1$--spines of minimal length in closed hyperbolic manifolds, and ultimately his arguments prove that either the $1$--spine has a bounded-length embedded cycle $C$, or the map is injective on the level of fundamental groups. See also \cite[Proposition~A.2]{Biringer}, from which the same conclusion can be drawn. When the essential systole is larger than the length of $C$, one obtains a contradiction similar to one seen in the proof above.  This is not explicitly stated, and rather than repeat White's arguments to deduce the result, the authors opted to apply the argument above.    Finally, we note that the alternate proof for Theorem~\ref{thm:main} in \S\ref{sec:conclusion} appeals to the more general version of Theorem~\ref{thm:KW} from \cite{kapovichweidmann}.

\subsection{Full rank--$n$ filling links}

Nowhere in the proof of Theorem~\ref{thm:big essential systole filling} was surjectivity of $i_* \circ f_*$ used.  This suggests the following definition.

Say that a map $f \colon \Gamma \to M$ has {\em full rank} if $\rank(f_*(\pi_1(\Gamma)) = \rank(\pi_1(\Gamma))$.  Then define a link $L \subset M$ to be {\em full rank--$n$ filling} if for every $f \colon \Gamma \to M \sss L$ such that $\rank(\pi_1(\Gamma)) \leq n$ and $i \circ f \colon \Gamma \to M$ is full rank, then $f_*$ is injective.  The proof of Theorem~\ref{thm:big essential systole filling} applies verbatim to prove the following.

\begin{theorem} \label{thm:fullrank1}
    Given $n >0$, there exists $R > 0$ such that if $L \subset M$ is a hyperbolic link with ${\sys}_e(L) > R$, then $L$ is full rank--$n$ filling. \qed
\end{theorem}

Combining Theorem~\ref{thm:unbounded essential systole} with Theorem~\ref{thm:big essential systole filling} implies Theorem~\ref{thm:main}, as described in the introduction. Similarly, combining it with Theorem~\ref{thm:fullrank1} proves the following.

\begin{theorem} \label{thm:fullrank2}
For every $n > 0$ and every closed, orientable $3$--manifold $M$, there is a full rank--$n$ filling link $L \subset M$. \qed
\end{theorem}

\begin{remark}
    The above strengthening of Theorem~\ref{thm:big essential systole filling}, as well as Theorem~\ref{thm:fullrank2} strengthening Theorem~\ref{thm:main}, were suggested to the authors by Ian Biringer.
\end{remark}

\section{Links with large essential systole} \label{S:links big esssys}

The goal of this section is to prove the following.

\medskip

\noindent
{\bf Theorem~\ref{thm:unbounded essential systole}}  {\em \UnbddEss}

\medskip

The proof involves constructing a particular tangle in a tetrahedron which, when removed from the tetrahedra ($3$--simplices) of a triangulation of a $3$--manifold, produces a link.  We apply this construction to a particularly nice kind of triangulation on $M$, and show that the resulting link is hyperbolic, and furthermore, if the ``combinatorial systole" of the triangulation is large, then the essential systole of the link is large.

\subsection{Triangulations} \label{Sec:Triangulations}
Suppose $M$ is a closed $3$--manifold, and $\tau$ is a triangulation of $M$.  Let $\hat \tau$ denote the subgraph of the $1$--skeleton of the first barycentric subdivision of $\tau$ which is the union of $1$--simplices adjacent to barycenters $v_T$ of the tetrahedra $T$ of $\tau$.  The graph $\hat \tau$ is bipartite, with the barycenters of tetrahedra one color and all other vertices the second color.  We equip $\hat \tau$ with the usual graph metric which assigns length $1$ to each edges.  Then for every two tetrahedra $T$ and $T'$ of $\tau$, the distance between $v_T$ and $v_{T'}$ is $2$ if and only if $T$ and $T'$ are distinct and share a vertex, edge, or face. 

The following lemma allows us to approximate essential loops in $3$--manifolds by {\em combinatorial} loops in the $1$-skeleton of $\hat{\tau}$ (i.e.~loops that are concatenations of finitely many edges). 

\begin{lemma} \label{lemma:combloop} Any non-null-homotopic loop $\gamma$ in $M$ is homotopic to a combinatorial loop $\gamma'$ in $\hat{\tau}$ such that $\gamma$ and $\gamma'$ meet the same set of tetrahedra in $\tau$. 
\end{lemma}

\begin{proof} We construct an open cover $\mathcal{U}$ for $M$ consisting of the open stars of simplices.  The open sets are thus interiors of the tetrahedra in $\tau$, as well as neighborhoods of the interiors of faces, edges, and vertices, which only meet the simplices in which they are contained. 

Suppose $\gamma \colon [0,1] \to M$ is a loop, and let $\epsilon > 0$ be the Lebesgue number of the cover $\gamma^{-1}(\mathcal U)$.  Let $n$ be a positive integer with $\frac1n < \epsilon$, so that $\gamma([\tfrac{j}n,\tfrac{j+1}n])$ is entirely contained in the open star of some simplex $\sigma_j$, for each $j=0,\ldots,n-1$.  Consequently, $\gamma(\tfrac{j}n)$ is contained in the intersections of the open stars of $\sigma_{j-1}$ and $\sigma_j$ (all indicies taken mod $n$).  There is a unique simplex $\sigma'_j$ that contains $\gamma(\tfrac{j}n)$ in its interior (for which $\sigma_{j-1}$ and $\sigma_j$ are faces).  Since open stars are contractible, for each interval, $[\tfrac{j}n,\tfrac{j+1}n]$, we may homotope the restriction of $\gamma$ to a map $\gamma'_j$, defined on the same interval, so that the endpoints are sent to the barycenters of $\sigma'_j$ and $\sigma'_{j+1}$.  We further assume (as we may) that (1) $\gamma'$ meets precisely the same set of tetrahedra as the restriction of $\gamma$, and (2) the homotopies are straight-line homotopies at the endpoints.  The second condition ensures that the homotopies glue together to give a homotopy from $\gamma$ to a combinatorial loop $\gamma'$ in $\hat \tau$, and the first condition ensures that $\gamma'$ meets precisely the same tetrahedra as $\gamma$.
\end{proof}




We define the \textit{combinatorial systole} of $\hat \tau$ to be: 
$$\text{combsys}(\hat \tau) = \min\{\ell_{\hat \tau}(\gamma): \gamma \text{ a loop in $\hat \tau$ which is not nullhomotopic in $M$}\}.$$

We will ultimately be interested in manifolds with triangulations with large combinatorial systole.  The main consequence of this for our purposes is the following.
\begin{lemma} \label{lem:big combsys, many tetrahedra}
If $\tau$ is any triangulation of a closed $3$--manifold $M$, then any non-null-homotopic loop in $M$ has non-empty intersection with least $n = \lfloor \tfrac{\combsys(\hat \tau)}{16} \rfloor$ tetrahedra, $T_0,\ldots,T_{n-1}$ of $\tau$ such that $d_{\hat \tau}(T_i,T_j) \geq 8|j-i|$ for all $1 \leq i,j \leq n$.
\end{lemma}

\begin{proof}  
By \Cref{lemma:combloop}, we may homotope $\gamma$ so that it is a combinatorial loop in the graph $\hat \tau$ meeting the exact same set of tetrahedra. We will find the required tetrahedra from those whose barycenters are vertices of $\gamma$. 

Take any closed ball in $\hat \tau$ of radius $R = \lfloor \frac{\combsys(\hat \tau)}2 \rfloor-2$ centered at any vertex $v \in \hat \tau$. We claim that its inclusion into $M$ must be trivial on $\pi_1$. To see this, we first construct a maximal tree in the ball of radius $R$ about $v$ as follows.  Enumerate the vertices in the ball, starting with $v$, and recursively define the tree as an increasing union of trees
\[ t_0 \subset t_1 \subset \ldots \subset t = \bigcup_k t_k.\]
The tree $t_0$ is just the vertex $\{v\}$.  Assuming we have constructed $t_k$, we construct $t_{k+1}$ from $t_k$ by adding the next vertex $v'$ in the enumeration that is not already included in $t_k$, and adding a shortest path between $v$ and $v'$.  The shortest path we add is required to follow the $t_k$ at the first vertex of $t_k$ it encounters (which by induction is indeed a shortest path).  The union of these nested trees is a tree $t$ which contains every vertex, so it is a maximal tree in the ball of radius $R$.  By construction, every vertex of $t$ is connected to $v$ by a path of length at most $R$.  The fundamental group is generated by loops of length at most $2R+1 < \combsys(\hat \tau)$ based at $v$ (one for every edge of the ball of radius $R$ not in $t$). These loops are all trivial by definition of the combinatorial systole, so the image of the fundamental group of the ball is trivial.

Now pick any vertex $v_{T_0}$ on $\gamma$ which is the barycenter of a tetrahedron $T_0$ and let $n$ be as in the lemma.  Since $\gamma$ is non-null-homotopic in $M$, it follows that it is not contained in the closed ball centered at $v_{T_0}$ of radius
\[ 8(n-1) = 8 \left \lfloor \frac{\combsys(\hat \tau)}{16}\right \rfloor - 8 < R.  \]
In particular, $\gamma$ nontrivially intersects the spheres of radius $8k$ centered at $v_{T_0}$, where $k=0,\ldots,n-1$, and we let $v_k \in \gamma$ be any vertex of intersection.  Since we are considering spheres of even radius, and a vertex in $\hat{\tau}$ whose distance from $v_{T_0}$ is even must also be a barycenter of a tetrahedron, it follows that $v_k = v_{T_k}$ for some tetrahedron $T_k$ for each $k$.
By construction, for all $i$ and $j$, $v_{T_i}$ and $v_{T_j}$ are distance at least $8|j-i|$ apart, proving the lemma. 
\end{proof}

We are interested in triangulations with large combinatorial systole to which we can apply the previous lemma. 
We will also want to impose some additional control on the local structure of our triangulations.  This additional control can be described by constraints on the links of vertices, which we now describe.

A {\em Cooper-Thurston triangulation} $\tau$ of a $3$--manifold is one for which the link of every vertex is isomorphic to the double over the boundary of one of the five triangulations of a disk shown in Figure~\ref{Fig:CT links}.  In \cite{CoopThur1988}, Cooper and Thurston proved that such triangulations exist.  A minor modification of their construction proves the next proposition.

\begin{figure}[h]
\begin{tikzpicture}
    \filldraw[opacity=.1] (0,0) -- (2,0) -- (2,2) -- (0,2) -- (0,0) -- (2,0);
    \draw[thick] (0,0) -- (2,0) -- (2,2) -- (0,2) -- (0,0) -- (2,0);
    \draw[thick] (0,0) -- (2,2);
    \draw[thick] (2,0) -- (0,2);
    \begin{scope}[shift = {(4,0)},scale = 1.1]
    \filldraw[opacity=.1] (0,0) -- (2,0) -- (1,1.7) -- (0,0) -- (2,0);
    \draw[thick] (0,0) -- (2,0) -- (1,1.7) -- (0,0) -- (2,0);
    \draw[thick] (1,0) -- (.5,.85) -- (1.5,.85) -- (1,0);
    \draw[thick] (0,0) -- (1.5,.85);
    \draw[thick] (2,0) -- (.5,.85);
    \draw[thick] (1,0) -- (1,1.7);
    \end{scope}
    \begin{scope}[shift = {(8,0)},scale = 1.1]
    \filldraw[opacity=.1] (0,0) -- (2,0) -- (1,1.7) -- (0,0) -- (2,0);
    \draw[thick] (0,0) -- (2,0) -- (1,1.7) -- (0,0) -- (2,0);
    \draw[thick] (0,0) -- (1.5,.85);
    \draw[thick] (2,0) -- (.5,.85);
    \draw[thick] (1,0) -- (1,1.7);
    \draw[thick] (1,.25) -- (1.27,.41) -- (1.27,.72) -- (1,.89) -- (.73,.72) -- (.73,.41) -- (1,.25);
    \draw[thick] (0,0) -- (.73,.72) -- (1,1.7);
    \draw[thick] (2,0) -- (1.27,.72) -- (1,1.7);
    \draw[thick] (0,0) -- (1,.25) -- (2,0);
    \end{scope}
    \begin{scope}[shift = {(2,-3)}]
    \filldraw[opacity=.1] (0,0) -- (2,0) -- (2,2) -- (0,2) -- (0,0) -- (2,0);
    \draw[thick] (0,0) -- (2,0) -- (2,2) -- (0,2) -- (0,0) -- (2,0);
    \draw[thick] (0,0) -- (2,2);
    \draw[thick] (2,0) -- (0,2);
    \draw[thick] (1,0) -- (1,2);
    \draw[thick] (0,1) -- (2,1);
    \draw[thick] (0,0) -- (1.5,.5) -- (2,2);
    \draw[thick] (2,0) -- (.5,.5) -- (0,2);
    \draw[thick] (0,0) -- (.5,1.5) -- (2,2);
    \draw[thick] (2,0) -- (1.5,1.5) -- (0,2);
    \end{scope}
    \begin{scope}[shift = {(6,-3)},scale=1.3]
    \filldraw[opacity=.1] (.4,0) -- (1.6,0) -- (2,1.1) -- (1,2) -- (0,1.1) -- (.4,0) -- (1.6,0);
    \draw[thick] (.4,0) -- (1.6,0) -- (2,1.1) -- (1,2) -- (0,1.1) -- (.4,0) -- (1.6,0);
    \draw[thick] (.4,0) -- (2,1.1);
    \draw[thick] (2,1.1) -- (0,1.1);
    \draw[thick] (0,1.1) -- (1.6,0);    
    \draw[thick] (1.6,0) -- (1,2);    
    \draw[thick] (1,2) -- (.4,0);
    
    \draw[thick] (.2,.55) -- (2,1.1);
    \draw[thick] (1,0) -- (1,2);
    \draw[thick] (0,1.1) -- (1.8,.55);
    \draw[thick] (.4,0) -- (1,.8) -- (1.59,1.47);
    \draw[thick] (1.6,0) -- (1,.8) -- (.41,1.47);
    \end{scope}
\end{tikzpicture}
\caption{The five allowable triangulations of links of vertices are obtained by doubling each of these triangulations of a disk.}
\label{Fig:CT links}
\end{figure}
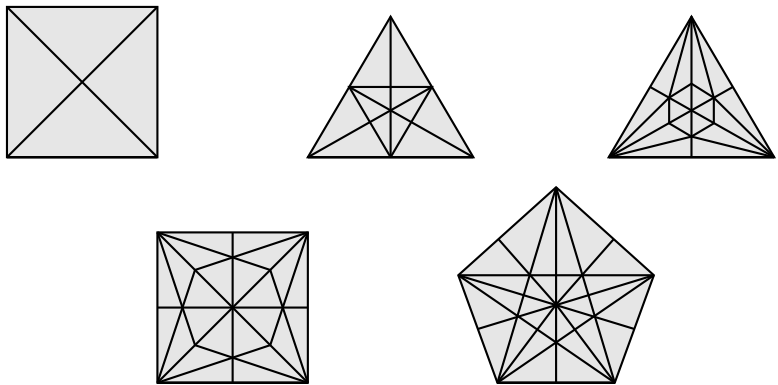

\begin{proposition} \label{Prop:big CT exist}
Suppose $M$ is a closed, orientable $3$--manifold $M$ and $r > 0$.  Then $M$  admits a {\em Cooper-Thurston} triangulation $\tau$ with $\combsys(\hat \tau) > r$.
\end{proposition}

\begin{proof} Cooper and Thurston's proof starts with a \textit{paving} of $M$, which is a subdivision of $M$ into $3$--dimensional cubes so that two such cubes either meet at a vertex, edge, or face, or are disjoint. Given an edge $e$ in a cube, the \textit{degree} of $e$ is the number of cubes that have $e$ as an edge. Cooper and Thurston prove there exists a paving of $M$ so that each edge has degree $3$, $4$, or $5$, and the degree $3$ and $5$ edges form disjoint embedded 1-manifolds. They then triangulate each cube to ensure the correct links (see \Cref{Fig:CT cube}). Note that the triangulation of each cube consists of three kinds of edges: the purple edges, which connect a vertex of the cube to the centers of the three faces in which the vertex is contained; the blue edges, which connect the vertices of the cube to the center of the cube; and the red edges, which connect the center of a face to the center of the cube.  Since the resulting triangulation $\tau$ depends on the paving $P$ of $M$, we will denote it $\tau(P)$. As before, $\hat{\tau}(P)$ will be the graph associated with $\tau(P)$ equipped with the length function $\ell_{\hat{\tau}(P)}$ induced by the usual graph metric.

\begin{figure}
    \centering
\includegraphics[width=0.5\linewidth]{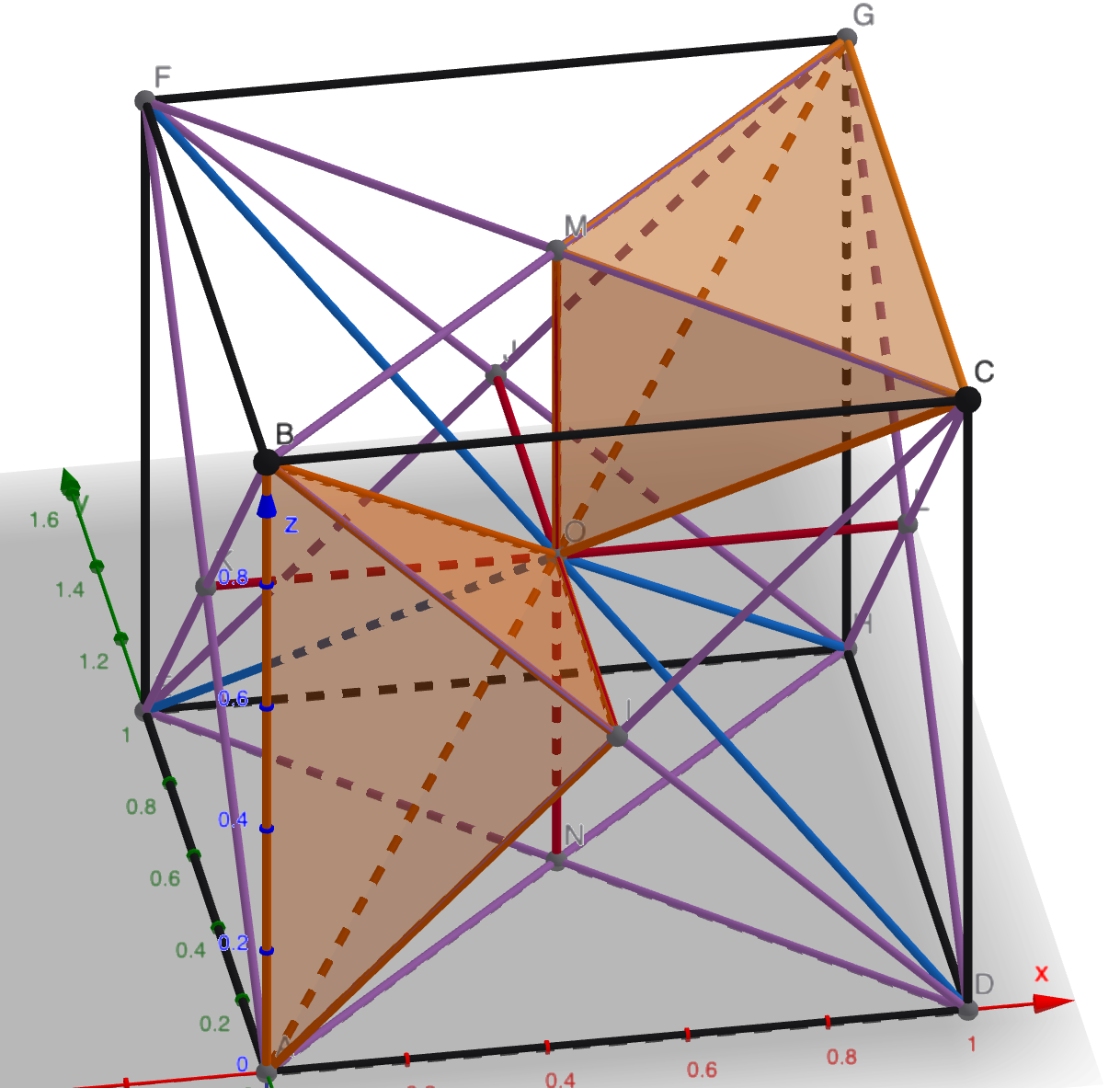}
    \caption{An illustration of the Cooper-Thurston triangulation on a single cube. The cube is triangulated by $24$ tetrahedra, two of which are shown in orange.}
    \label{Fig:CT cube}
\end{figure}

Cooper and Thurston show there is a paving $P$ of $M$ so that the triangulation of each cube as above results in an honest triangulation $\tau(P)$.  For any integer $k \geq 1$, we can subdivide each cube of $P$ into $k^3$ sub-cubes, producing a new paving $P_k$. Triangulating each of the sub-cubes as above results in another Cooper-Thurston triangulation of $M$.  (In fact, the construction in \cite{CoopThur1988} has such a subdivision of the paving built into it.) We will show that for any $r > 0$, we can choose some $k(r) > 0$ sufficiently large so that for the paving $P_{k(r)}$, we have combsys\big($\hat{\tau}(P_{k(r)})\big) > r$.

Given a paving $P$ for which $\tau(P)$ is a Cooper-Thurston triangulation, we can define a singular, locally Euclidean geodesic metric on $M$ in which each cube of $P$ is (locally) isometric to a unit cube in $\mathbb R^3$; that is, a cube where all side lengths equal $1$.  For any path $\alpha$ in $M$, we write $\ell_{P}(\alpha)$ to denote its length in this metric, due to its dependence on $P$. 

Since each edge $e$ of $\hat \tau(P)$ is a path in $M$, we can measure its length, $\ell_P(e)$.  Here we derive an upper bound on $\ell_P(e)$ for any edge $e$ of $\hat{\tau}(P)$. For this, we assume the cube is embedded in $\mathbb R^3$ with vertices at the points
\[ \{(\varepsilon_1,\varepsilon_2,\varepsilon_3) \mid \varepsilon_i \in \{0,1\} \mbox{ for } i =1,2,3 \}.\]
Then the center of $C$ is $\left(\frac{1}{2}, \frac{1}{2}, \frac{1}{2}\right)$, and the center of one of the faces is $\left(\frac{1}{2}, 0, \frac{1}{2}\right)$. One then obtains that the barycenter of the tetrahedon $T$ with vertex set
\[ \left\{(0, 0, 0), (0, 0, 1), \left(\frac{1}{2}, 0, \frac{1}{2}\right), \left(\frac{1}{2}, \frac{1}{2}, \frac{1}{2}\right)\right\}\]
is $\left(\frac{1}{4}, \frac{1}{8}, \frac{1}{2}\right)$. We then have that the maximum distance between a vertex of $T$ and $\left(\frac{1}{4}, \frac{1}{8}, \frac{1}{2}\right)$ is equal to $\frac{\sqrt{21}}{8}$, which is an upper bound for the edge lengths of $\hat{\tau}(P)$. 

When every cube $C$ in a paving $P$ is subdivided into $k^3$ cubes to produce the new paving $P_k$, for any curve $\gamma$, $\ell_{P_k}(\gamma) = k\ell_{P}(\gamma)$. We can see this most clearly when $\gamma$ is an edge of $C$. In the new metric, $s_C$, an edge of $C$, will pass through $k$ cubes each of side length equal to $1 = \ell_{P}(s_C)$, so $\ell_{P_k}(s_C) = k = k\ell_{P}(s_C)$.  

Consider any paving $P$ so that $\tau(P)$ is a Cooper-Thurston triangulation, as before. Let $\gamma$ be a non-null-homotopic curve in $M$ which minimizes $\ell_P(\gamma)$ among all such closed curves $\gamma$.
Find $k(r)$ so that the paving $P_{k(r)}$ obtained by subdividing each cube in $P$ into $(k(r))^3$ cubes will yield $\ell_{P_{k(r)}}(\gamma) > r$, which is made possible by the discussion in the previous paragraph. Observe that $\gamma$ will still minimize length among non-null-homotopic curves with respect to the new metric induced by $P_{k(r)}$, as the new metric is the old metric scaled by a factor of $k(r)$. 

Consider the combinatorial systole of $\hat{
\tau}(P_{k(r)})$, which we will call $\gamma'$. Note that for each edge $e \in \gamma'$, $\ell_{P_{k(r)}}(e) \leq \frac{\sqrt{21}}{8}$ by construction. Thus,
\begin{align*}\text{combsys}(\hat{\tau}(P_{k(r)})) & =  \ell_{\hat{\tau}(P_{k(r)})}(\gamma') = (\text{\# of edges in $\gamma'$})(1)\\
& \geq (\text{\# of edges in $\gamma'$})\bigg(\frac{\sqrt{21}}{8}\bigg)\\
&\geq  \sum\limits_{e \in \gamma'} \ell_{P_{k(r)}}(e) \geq \ell_{P_{k(r)}}(\gamma')\\
& \geq \ell_{P_{k(r)}}(\gamma) > r. \end{align*}
The second to last inequality follows since $\gamma$ minimized length over all non-null-homotopic curves, and $\gamma'$ is non-null-homotopic.
\end{proof}

We will also need the following, which is immediate by inspecting Figure~\ref{Fig:CT links}.
\begin{lemma} \label{Lem:Flag + edge degree}
    Suppose $\tau$ is a Cooper-Thurston triangulation of a closed $3$--manifold $M$.  Then every link of a vertex is flag, and every edge has degree $4$, $6$, $8$, or $10$. \qed
\end{lemma}

\begin{remark}
Brady-McCammond-Meier \cite{BraMcCMei2004} describe another construction of triangulations with related bounds on the combinatorics.  It seems likely that the barycentric subdivisions of these triangulations could also be made to work for our purposes below, but constructing Cooper-Thurston triangulations with large combinatorial systole is likely easier.
\end{remark}

\subsection{A hyperbolic tetrahedron tangle} \label{Sec:Tet tangle}

Let $T_0$ be a tetrahedron and $L_0 \subset T_0$ be the tangle in $T_0$, which is the union of the four embedded circles and six properly embedded arcs shown on the left in Figure~\ref{Fig:link-tangle tet}.  Explicitly, we assume $T_0$ is a regular, Euclidean tetrahedron with side lengths $1$ and
\begin{enumerate}
    \item each circle has radius $1/4$ and is centered at the barycenter of the face, bounding a disk, and 
    \item each arc is the intersection with $T_0$ of a circle of radius $1/8$ centered on the barycenter of an edge and contained in a plane orthogonal to the edge.
\end{enumerate}
From the assumptions, the tangle meets each face as shown on the right in Figure~\ref{Fig:link-tangle tet}, and the endpoints of the arcs are inside the disks bounded by the circles in the faces.
\begin{remark}
    Our use of the term ``tangle" may be slightly non-standard, but we will only use it in reference to the specific embedded $1$--manifold in $T_0$.
\end{remark}
\begin{center}
\begin{figure}[h]
\begin{tikzpicture}
    \node at (0,0) {\includegraphics[width=5cm]{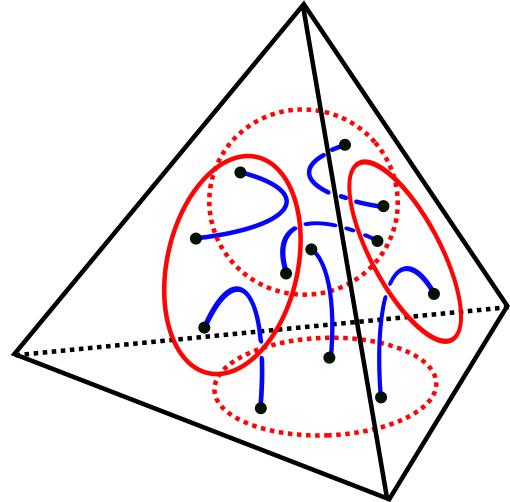}}; 
    \draw[thick] (5,-1.5) -- (9,-1.5) -- (7,1.96) -- (5,-1.5) -- (9,-1.5);
    \draw[thick] (7,-.345) circle (1);
    \filldraw (7,-1) circle (.05);
    \filldraw (6.42,-.01) circle (.05);
    \filldraw (7.58,-.01) circle (.05);
\end{tikzpicture}
\caption{{\bf Left:} The tangle $L_0 \subset T_0$. {\bf Right:} The intersection of $L_0$ with the face.}
\label{Fig:link-tangle tet}
\end{figure}
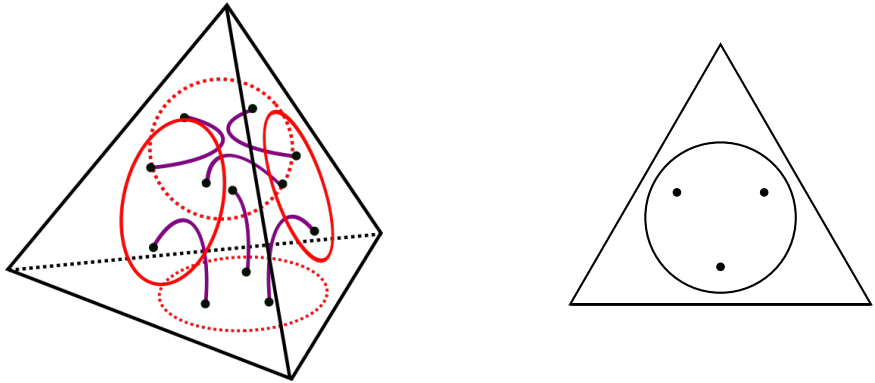
\end{center}

By construction, the full symmetry group $\Tet \cong S_4$ of $T_0$ acts on $T_0$ by isometries preserving $L_0$.  We can take a fundamental domain,  $\Delta$, in $T_0$ for $\Tet$ to be (any) $3$--simplex of the first barycentric subdivision of $T_0$, and $L_0$ meets $\Delta$ in a pair of arcs contained in a pair of sides. The fundamental domain $\Delta$ and pair of arcs are illustrated in Figure~\ref{Fig:fundamental piece}.

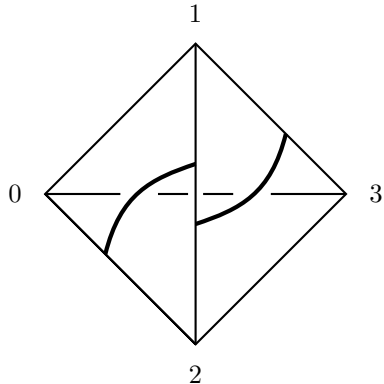
\begin{figure}[h]
\begin{tikzpicture}[scale=2]
    \draw[thick] (0,0) -- (1,-1) -- (2,0) -- (1,1) -- (0,0) -- (1,-1);
    \draw[thick] (0,0) -- (.5,0);
    \draw[thick] (.75,0) -- (.95,0);
    \draw[thick] (1.05,0) -- (1.25,0);
    \draw[thick] (1.5,0) -- (2,0);
    \draw[thick] (1,-1) -- (1,1);
    \draw[ultra thick, red] (1,.2) .. controls (.7,.1) and (.5,0) .. (.4,-.4);
    \draw[ultra thick, blue] (1,-.2) .. controls (1.3,-.1) and (1.5,0) .. (1.6,.4);
    \node at (2.2,0) {$3$};
    \node at (1,-1.2) {$2$};
    \node at (1,1.2) {$1$};
    \node at (-.2,0) {$0$};
\end{tikzpicture}
\caption{The fundamental domain in $T_0$ for the action of $\Tet$.  The vertices are barycenters of a vertex, edge, face, and of $T_0$, and are labeled here by $0$, $1$, $2$, and $3$, respectively.}
\label{Fig:fundamental piece}
\end{figure}

We consider $T_0 \sss L_0$ as a manifold with corners.  The corners are at the $1$--skeleton of $T_0$, and the boundary consists of the union of the faces minus $L_0$, which are each thus homeomorphic to a $2$--simplex minus a circle and three points in the disk bounded by the circle, as on the right in Figure~\ref{Fig:link-tangle tet}.

\begin{proposition} \label{prop:explicit hyp tet}
    There is a complete hyperbolic structure on $T_0 \sss L_0$ of finite volume, such that all boundary components are totally geodesic and all dihedral angles at the corners are $\frac{\pi}2$.
\end{proposition}
\begin{proof} We construct a hyperbolic structure on the fundamental domain $\Delta \sss L_0$ explicitly.  To do this, we we will find a complete, finite volume hyperbolic structure on $\Delta \sss L_0$ with totally geodesic boundary, whose dihedral angles along the corners are as illustrated on the left of Figure~\ref{Fig:hyperbolic polyhedron 1}.  By collapsing the two arcs to points, it suffices to find a partially ideal hyperbolic polyhedron as illustrated on the right of Figure~\ref{Fig:hyperbolic polyhedron 1}, where the two ``dots" are ideal (hence deleted).

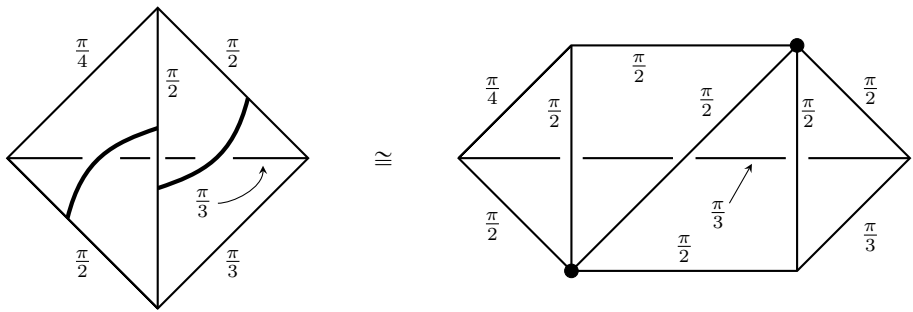
\begin{figure}[h]
\begin{tikzpicture}
    \begin{scope} [scale = 2]
        \draw[thick] (0,0) -- (1,-1) -- (2,0) -- (1,1) -- (0,0) -- (1,-1);
        \draw[thick] (0,0) -- (.5,0);
        \draw[thick] (.75,0) -- (.95,0);
        \draw[thick] (1.05,0) -- (1.25,0);
        \draw[thick] (1.5,0) -- (2,0);
        \draw[thick] (1,-1) -- (1,1);
        \draw[ultra thick, red] (1,.2) .. controls (.7,.1) and (.5,0) .. (.4,-.4);
        \draw[ultra thick, blue] (1,-.2) .. controls (1.3,-.1) and (1.5,0) .. (1.6,.4);
        \node at (2.5,0) {$\cong$};
        \node at (.5,-.7) {$\frac{\pi}2$};
        \node at (1.1,.5) {$\frac{\pi}2$};
        \node at (1.5,.7) {$\frac{\pi}2$};
        \node at (1.5,-.7) {$\frac{\pi}3$};
        \node at (.5,.7) {$\frac{\pi}4$};
        \node at (1.3,-.3) {$\frac{\pi}3$};
        \draw[->,>=stealth] (1.4,-.3) .. controls (1.5,-.3) and (1.7,-.2) .. (1.7,-.05);
    \node at (2.1,0) {$3$};
    \node at (1,-1.1) {$2$};
    \node at (1,1.1) {$1$};
    \node at (-.1,0) {$0$};
    \end{scope}
    \begin{scope}[shift = {(6,0)},scale=1.5]
        \draw [thick] (0,0) -- (1,1) -- (3,1) -- (4,0) -- (3,-1) -- (1,-1) -- (0,0) -- (1,1);
        \draw [thick] (1,-1) -- (3,1);
        \draw [thick] (1,1) -- (1,-1);
        \draw [thick] (3,1) -- (3,-1);
        \draw [thick] (0,0) -- (.9,0);
        \draw [thick] (1.1,0) -- (1.9,0);
        \draw [thick] (2.1,0) -- (2.9,0);
        \draw [thick] (3.1,0) -- (4,0);
        \filldraw [red] (1,-1) circle (.06);
        \filldraw [blue] (3,1) circle (.06);
        \node at (.3,-.6) {$\frac{\pi}2$};
        \node at (.85,.4) {$\frac{\pi}2$};
        \node at (1.6,.8) {$\frac{\pi}2$};
        \node at (2,-.8) {$\frac{\pi}2$};
        \node at (2.2,.5) {$\frac{\pi}2$};
        \node at (3.1,.4) {$\frac{\pi}2$};
        \node at (3.65,.6) {$\frac{\pi}2$};
        \node at (3.65,-.7) {$\frac{\pi}3$};
        \node at (.3,.6) {$\frac{\pi}4$};
        \node at (2.3,-.5) {$\frac{\pi}3$};
        \draw[->,>=stealth] (2.4,-.4) -- (2.6,-.05);
    \node at (4.15,0) {$3$};
    \node at (3,-1.15) {$2$};
    \node at (1,1.15) {$1$};
    \node at (-.15,0) {$0$};
        \end{scope}
\end{tikzpicture}
\caption{{\bf Left:} Dihedral angles for hyperbolic structure on $\Delta \sss L_0$. {\bf Right:} Partially ideal polyhedron $P$ with ideal vertices obtained by collapsing arcs to a point (illustrated by a dot). Vertex numbers and labels match Figure~\ref{Fig:fundamental piece}.}
\label{Fig:hyperbolic polyhedron 1}
\end{figure}

We can construct such a polyhedron $P \subset \mathbb H^3$ explicitly as the intersection of hyperbolic half-spaces in the upper half space model $\mathbb H^3 = \{(x,y,z) \mid z > 0 \}$.  In these coordinates, the polyhedron is given by the set of points $(x,y,z) \in \mathbb H^3$ satisfying
\begin{enumerate}
    \item $0 \leq x \leq 3 +\frac{3\sqrt{2}}{2}$,
    \item $0 \leq y \leq 1 +\frac{\sqrt{2}}2$, 
    \item $x^2+(y-1)^2 + z^2 \geq 1$, and
    \item $(x-(2+\sqrt{2}))^2 + y^2 + z^2 \geq (2+\sqrt{2})^2.$
\end{enumerate}
One ideal vertex is at $\infty$ and the other is at $(0,0,0)$.  Figure~\ref{Fig:hyperbolic polyhedron 2} shows the polyhedron viewed from the vertex at infinity.  

\begin{figure}[h]
\begin{tikzpicture}
    \begin{scope}
    \draw [->,>=stealth] (-1,0) -- (7,0);
    \draw [->,>=stealth] (0,-.6) -- (0,3.5);
    \filldraw [opacity = .2] (0,0) -- (5.1213,0) -- (5.1212,1.7071) -- (0,1.7071) -- (0,0);
    \draw [ultra thick] (0,0) -- (5.1213,0) -- (5.1212,1.7071) -- (0,1.7071) -- (0,0);
    \draw [ultra thick] (0,0) -- (.52,1.7071);
    \draw[thick,dotted] [domain=0:360] plot ({cos(\x)}, {1+sin(\x)});
    \draw[thick,dotted] [domain=-10:190] plot ({3.4142+3.4142*cos(\x)},{3.4142*sin(\x)});
    \node at (-.2,1.75) {$1$};
    \node at (.65,1.45) {$0$};
    \node at (5.3,1.75) {$3$};
    \node at (5.1,-.2) {$2$};
    \filldraw [red] (0,0) circle (.1);
    \end{scope}
\end{tikzpicture}
\caption{View of $P \subset \mathbb H^3$ from above with red vertex and numbers matching those from Figures~\ref{Fig:fundamental piece} and \ref{Fig:hyperbolic polyhedron 1}. The blue vertex is at infinity.}
\label{Fig:hyperbolic polyhedron 2}
\end{figure}
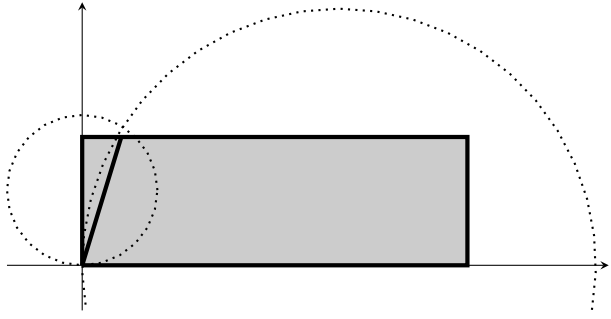

The hyperbolic half-spaces are bounded by hyperbolic planes that meet the sphere at infinity in four lines and two circles.  The equations defining these lines and circles are obtained by making the inequalities above into equations, and setting $z=0$.
One can directly check that the lines and circles intersect in the required angles, and hence so do the hyperbolic planes.
\end{proof}

\subsection{Construction of links and proof of Theorem~\ref{thm:unbounded essential systole}} \label{S:final construction}

Given a triangulation $\tau$ of a closed $3$--manifold, assume that each simplex is regular Euclidean with side length $1$ and the face gluings are by isometries. The copies of $L_0$ in each tetrahedron match up to define a link we denote $L_{\tau} \subset M$.  That is, $L_\tau \subset M$ is a link such that for every tetrahedron $T$ in $\tau$, the pair $(T,L_{\tau} \cap T)$ is homeomorphic to $(T_0,L_0)$.

The first fact we will need is the following.

\begin{lemma} \label{Lem:CT implies hyperbolic}
    If $\tau$ is a Cooper-Thurston triangulation of $M$, then $L_{\tau}$ is a hyperbolic link.
\end{lemma}
\begin{proof}
    Fix the hyperbolic structure on $T_0 \sss L_0$ from Proposition~\ref{prop:explicit hyp tet}.  Then $M \sss L_{\tau}$ can be obtained by gluing copies of the hyperbolic structure on $T_0 \sss L_0$  by isometries along the boundary, defining a finite volume, piecewise hyperbolic structure. Since the link of every vertex is flag by Lemma~\ref{Lem:Flag + edge degree}, each link is CAT($1$) (c.f.~\cite[Theorem II.5.18]{BridHaef1999}). Furthermore, since the dihedral angles of all corners are $\frac{\pi}2$, the metric on $M \sss L_{\tau}$ is locally $\CAT(-1)$; see \cite[Theorem~II.5.2]{BridHaef1999}.  
    Consequently, the universal cover is $\CAT(-1)$. Thus, by \Cref{T:CAT-1 to hyp}, it follows that $M \sss L_{\tau}$ is hyperbolic. 
\end{proof}

Continue to assume that $\tau$ is a Cooper-Thurston triangulation.  We let $\mathcal V_\tau$, $\mathcal E_\tau$, $\mathcal F_\tau$, and $\mathcal T_\tau$ denote the set of vertices, edges, faces, and tetrahedra of $\tau$.  We now describe a canonical collection of surfaces $\mathcal S_\tau$ associated to $\tau$.
This collection of surfaces is indexed by $\mathcal E_\tau \sqcup \mathcal F_\tau \sqcup \mathcal T_\tau$, as follows (see \Cref{fig:surfaces}):
    \begin{enumerate}
    \item For every $F \in \mathcal F_\tau$, $\Sigma_F$ is the disk bounded by the component of $L_\tau$ embedded in $F$.
    \item For every $E \in \mathcal E_\tau$, there is a component of $L_\tau$ that encircles $E$, built from a subset of the arcs of $L_\tau$ intersected with the tetrahedra containing $E$. The surface $\Sigma_E$ is the disk bounded by this link component.
    \item For every $T \in \mathcal{T}_{\tau}$, let $F_1,\ldots,F_4$ be the faces of $T$.  The surface $\Sigma_{T}$ is obtained by ``pushing" $T \sss \{\Sigma_{F_1} \cup \Sigma_{F_2} \cup \Sigma_{F_3} \cup \Sigma_{F_4}\}$ into the interior of $T$.
    \end{enumerate}

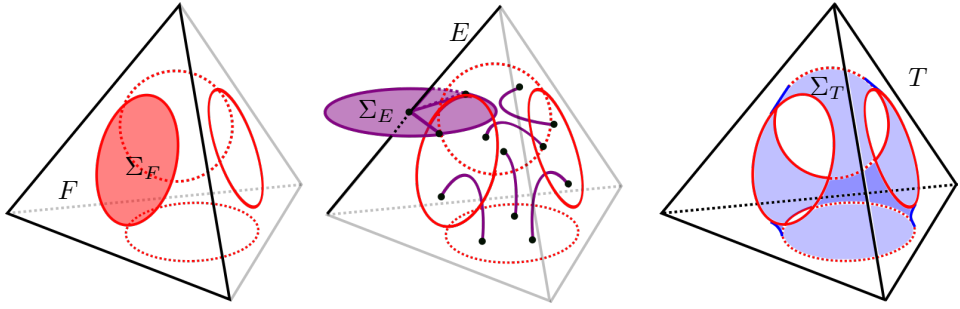
\begin{figure}[H]
    \centering
    \begin{tikzpicture}
    \node at (0,0) {\includegraphics[width=\linewidth]{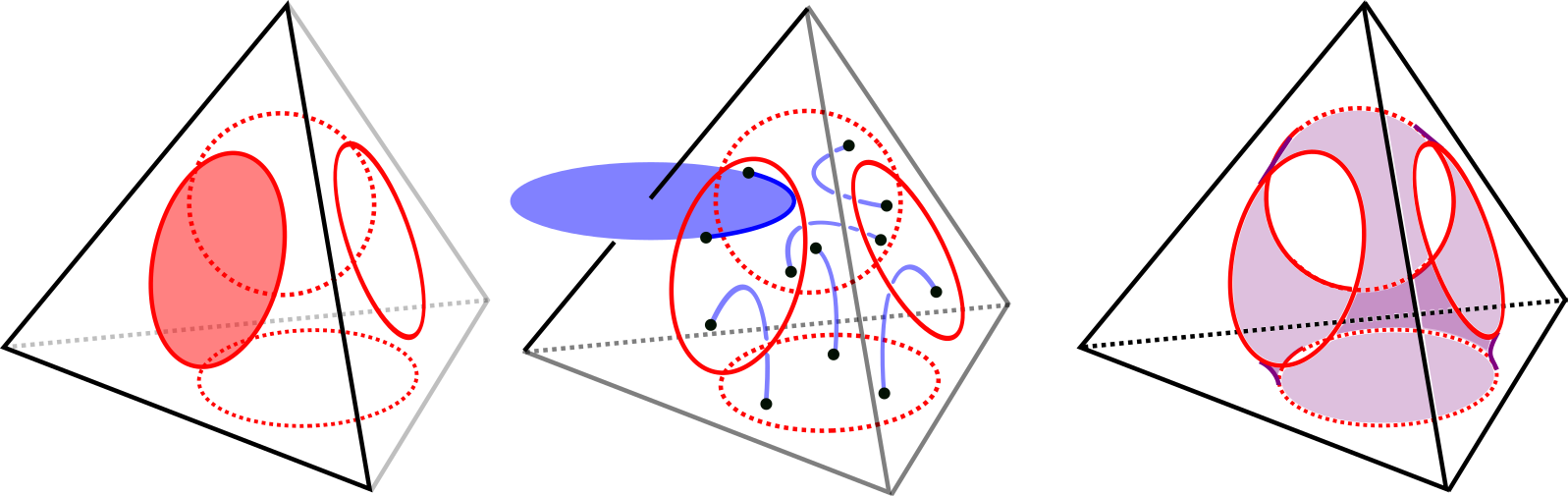}};
    \node at (-4.5,-.2) {$\Sigma_F$};
    \node at (-5.5,-.5) {$F$};
    \node at (-1.4,.9) {$\Sigma_E$};
    \node at (-.4,1.6) {$E$};
    \node at (4.6,.86) {$\Sigma_T$};
    \node at (5.8,1) {$T$};
    \end{tikzpicture}
    \caption{\textbf{Left:} A face $F$ of a tetrahedron and the associated surface $\Sigma_F$. \textbf{Middle:} An edge $E$ of a tetrahedron, and the associated surface $\Sigma_E$. \textbf{Right:} The four-holed sphere $\Sigma_T$ associated with a tetrahedron $T$. For clarity, $L_\tau$ is only partially shown in the left and right figures.}
    \label{fig:surfaces}
\end{figure}

Let $\Sigma_x^\circ$ denote the intersection of $\Sigma_x$ with $M \sss L_\tau$,
\[ \Sigma_x^\circ = \Sigma_x \cap M \sss L_\tau\]
for all $\Sigma_x \in \mathcal S_\tau$.  We also write
\[ \mathcal S_\tau^\circ = \{ \Sigma_x^\circ \mid \Sigma_x \in \mathcal S_\tau \}.\]
By inspection, we see that each $\Sigma_x^\circ \in \mathcal S_\tau^\circ$ is a punctured sphere.  More precisely, we have the following.
\begin{enumerate}
    \item For every $F \in \mathcal F_\tau$, $\Sigma_F^\circ$ is a four-punctured sphere.
    \item For every $E \in \mathcal E_\tau$, $\Sigma_E^\circ$ is a $(k+1)$--punctured sphere, where $k$ is the degree of the edge $E$.
    \item For every $T \in \mathcal T_\tau$, $\Sigma_T^\circ$ is a four punctured sphere.
\end{enumerate}
In particular, the number of punctures is uniformly bounded (at most 11) by Lemma~\ref{Lem:Flag + edge degree}.
\begin{remark}
    We emphasize that, except for $x = T$, $\Sigma_x^\circ$ is not the interior of $\Sigma_x$ since there are other components of $L_\tau$ that puncture $\Sigma_x$.  We also note that the surface $\Sigma_x$ we have described naturally intersects $L_\tau$ minimally in the isotopy class, rel $\partial \Sigma_x$ (which can be seen by considerations of the algebraic intersection number, for example).
\end{remark}

Recall that a surface $\Sigma$ with negative Euler characteristic which is properly embedded in a $3$-manifold $M$ is \textit{incompressible} if the inclusion $i: \Sigma  \hookrightarrow M$ induces an injective map between fundamental groups. An incompressible surface $\Sigma$ in a hyperbolic $3$--manifold is \textit{quasi--Fuchsian} if its fundamental group has a quasi-circle as its limit set for the action on the sphere at infinity of $\mathbb H^3$.

We now list some basic properties of the surfaces in $\mathcal{S}^{\circ}_{\tau}$, where $\tau$ is a Cooper-Thurston triangulation.

\begin{lemma} \label{Lem:totally geodesic CAT-1}
    Suppose $\tau$ is a Cooper-Thurston triangulation.  Then each $\Sigma_x^\circ \in \mathcal S_\tau^\circ$ is totally geodesic in the locally $\CAT(-1)$ metric on $M \sss L_\tau$.  In particular, each such $\Sigma_x^\circ$ is incompressible and quasi-Fuchsian in the hyperbolic structure.
\end{lemma}

\begin{proof}
We consider $\Sigma_x^\circ$ for $x \in \mathcal E_\tau \cup \mathcal F_\tau \cup \mathcal T_\tau$ separately, depending on whether $x$ is an edge, face, or tetrahedron.

For $x = F \in \mathcal F_\tau$ and $x = E \in \mathcal E_\tau$, the totally geodesic representatives are constructed explicitly from the hyperbolic structure on $T \cap (M\sss L_\tau) \cong T_0 \sss L_0$ for each tetrahedron $T \in \mathcal T_\tau$. Indeed, this structure is induced by the hyperbolic structure on $P$, the fundamental domain for the action of $\Tet$, as described in the proof of \Cref{prop:explicit hyp tet}. Thus, we can see that the totally geodesic representative of $\Sigma_F^\circ$ is built by gluing together copies of the face spanned by the vertex labeled $2$ (the barycenter of the face) and the blue and red ideal vertices (corresponding to a puncture and boundary component respectively) in Figure~\ref{fig:totally geodesic pieces}.  Similarly, the totally geodesic representative of $\Sigma_E^\circ$ is obtained by gluing together copies of the face spanned by the vertex labeled $1$ and the blue and red ideal vertices in the same figure.

\begin{figure}[h]
\begin{tikzpicture}
    \begin{scope} [scale = 2]
        \draw[thick] (0,0) -- (1,-1) -- (2,0) -- (1,1) -- (0,0) -- (1,-1);
        \draw[thick] (0,0) -- (.5,0);
        \draw[thick] (.75,0) -- (.95,0);
        \draw[thick] (1.05,0) -- (1.25,0);
        \draw[thick] (1.5,0) -- (2,0);
        \draw[thick] (1,-1) -- (1,1);
        \draw[ultra thick, red] (1,.2) .. controls (.7,.1) and (.5,0) .. (.4,-.4);
        \draw[ultra thick, blue] (1,-.2) .. controls (1.3,-.1) and (1.5,0) .. (1.6,.4);
        \filldraw[red,opacity=.15] (1,-1) -- (1,.2) .. controls (.7,.1) and (.5,0) .. (.4,-.4) -- (1,-1);
        \filldraw[blue,opacity=.15]
        (1,1) -- (1,-.2) .. controls (1.3,-.1) and (1.5,0) .. (1.6,.4) -- (1,1);
    \node at (2.1,0) {$3$};
    \node at (1,-1.1) {$2$};
    \node at (1,1.1) {$1$};
    \node at (-.1,0) {$0$};
    \node at (.75,-.35) {$\Sigma_F$};
    \node at (1.25,.35) {$\Sigma_E$};
    \end{scope}
    \begin{scope}[shift = {(6,0)},scale=1.5]
        \draw[thick, dashed] (1,-1) -- (1.8,1) -- (2.3,0) -- (1,-1);
        \filldraw[violet, opacity=.15] (1,-1) -- (1.8,1) -- (2.3,0) -- (1,-1);
        \filldraw[red,opacity=.15] (1,-1) -- (3,-1) -- (3,1) -- (1,-1);
        \filldraw[blue,opacity=.15]
        (1,-1) -- (3,1) -- (1,1) -- (1,-1);
        \draw [thick] (0,0) -- (1,1) -- (3,1) -- (4,0) -- (3,-1) -- (1,-1) -- (0,0) -- (1,1);
        \draw [thick] (1,-1) -- (3,1);
        \draw [thick] (1,1) -- (1,-1);
        \draw [thick] (3,1) -- (3,-1);
        \draw [thick] (0,0) -- (.9,0);
        \draw [thick] (1.1,0) -- (1.9,0);
        \draw [thick] (2.1,0) -- (2.9,0);
        \draw [thick] (3.1,0) -- (4,0);
        \filldraw [red] (1,-1) circle (.06);
        \filldraw [blue] (3,1) circle (.06);
    \node at (4.15,0) {$3$};
    \node at (3,-1.15) {$2$};
    \node at (1,1.15) {$1$};
    \node at (-.15,0) {$0$};
    \node at (1.3,.6) {$\Sigma_E^\circ$};
    \node at (2.5,-.5) {$\Sigma_F^\circ$};
    \node at (1.85,.3)
    {$\Sigma_T^\circ$};
        \end{scope}
\end{tikzpicture}
\caption{{\bf Left:} Pieces of $\Sigma_E$ and $\Sigma_F$ in $P$, the fundamental polyhedron for the action of $\Tet$. {\bf Right:} The totally geodesic representatives of $\Sigma_E^\circ$, $\Sigma_F^\circ$, and $\Sigma_T^\circ$ are constructed by gluing together the indicated totally geodesic polygons in $P$.} \label{fig:totally geodesic pieces}
\end{figure}


For $x = T \in \mathcal T_\tau$, we observe that $\Sigma_T
$ is entirely contained in $T$, so we may consider the case of $\Sigma_{T_0}^\circ \subset T_0\sss L_0$.  First we note that $\Sigma_{T_0}^\circ$ is incompressible: the boundary of a compressing disk would necessarily subdivide the four-holed sphere into two pairs of pants, and compressing would produce a pair of annuli between distinct cusps, which is impossible.  The orientation preserving subgroup $\Tet$ preserves (the isotopy class of) $\Sigma_{T_0}^\circ$, and the quotient of $\Sigma_{T_0}^\circ$ in $T_0 \sss L_0$ is an orbifold with one puncture and two cone points of order $2$ and $3$.  In particular, the orbifold must be totally geodesic  (see \cite[Theorem 2.1]{as}), and hence $\Sigma_{T_0}^\circ$ is totally geodesic.

Being totally geodesic in the locally $\CAT(-1)$ metric implies that the surfaces are incompressible.  The universal covers of the surfaces are isometrically embedded in the universal cover of $M \sss L_\tau$ with its $\CAT(-1)$ metric.  Since the identity on the universal covers is a quasi-isometry with respect to the $\CAT(-1)$ metric on the domain (see \Cref{T:CAT-1 to hyp}) and the hyperbolic metric on the range, it follows that the universal covers of the surfaces are quasi-isometrically embedded in $\mathbb H^3$.  In particular, their limit sets are quasi-circles (c.f.~\cite[Theorem III.H.3.9]{BridHaef1999}); hence, the surfaces are quasi-Fuchsian.
\end{proof}

The proof of \Cref{Lem:totally geodesic CAT-1} explicitly describes the totally geodesic representatives of $\Sigma_E^\circ$ and $\Sigma_F^\circ$, which are obtained by gluing together copies of faces of $P$, while the totally geodesic representative $\Sigma_T^\circ$ is only described ``abstractly".  We can also see the surface $\Sigma_T^\circ$ explicitly as follows.  From the proof, because the surface is invariant by $\Tet$, it must pass through the axis of the rotation of order $3$ from the vertex through the barycenter of the tetrahedron, as well as the axis of the rotation of order $2$ through the barycenters of the edge and the tetrahedron.  Since its boundary corresponds to the red cusp, it must be obtained by gluing together copies of the polygon illustrated in Figure~\ref{fig:totally geodesic pieces}.  


From this we easily deduce the following corollary, which asserts that the intersections of the totally geodesic surfaces in $\CAT(-1)$ metric in the isotopy classes of the surfaces in $\mathcal S_\tau^\circ$ intersect in the ``obvious" way.

\begin{corollary} \label{Cor:intersection pattern}
    The totally geodesic representatives of the surfaces in $\mathcal S_\tau^\circ$ with respect to the locally $\CAT(-1)$ metric have the following property: For any two distinct $\Sigma_x^\circ,\Sigma_y^\circ \in \mathcal S_\tau^\circ$, one of the following must occur:
    \begin{enumerate} 
    \item $\Sigma_x^\circ$ and $\Sigma_y^\circ$ are disjoint, or
    \item $\Sigma_x^\circ$ and $\Sigma_y^\circ$ cannot be isotoped to be disjoint. In addition, $\{x,y\} = \{E,F\}$ where $E \subset F$ or $\{x, y\} = \{E,T\}$ with $E \subset T$.  
    \end{enumerate}
    In Case (2), the surfaces intersect in a single arc. \qed
\end{corollary}

 We use these totally geodesic representatives to prove the next lemma. Before doing so, we recall that an intersection point $x$ of curve $\gamma$ with a properly embedded incompressible surface $S \subset N$ in a hyperbolic $3$--manifold $N$ is {\em essential} if there is a lift $\tilde{\gamma} \colon \mathbb R \to \mathbb H^3$ of $\gamma$ to the universal cover of $N$ that intersects a component $\widetilde{S}$ of the preimage of $S$ in a single point $\tilde x$ that projects to $x$.
If there is an essential intersection point between $\gamma$ and $S$, we say that they {\em intersect essentially}, and note that in this case, $\gamma$ and $S$ cannot be homotoped to be disjoint.

\begin{lemma} \label{Lem:snaking through}
    If $\gamma$ is a loop in $M \sss L_\tau$ which is non-null-homotopic in $M$, then $\gamma$ essentially intersects at least $n = \lfloor \tfrac{\combsys(\hat \tau)}{16} \rfloor$ surfaces
    \[ \Sigma_{x_0}^\circ,\ldots,\Sigma_{x_{n-1}}^\circ \in \mathcal S_\tau^\circ.\] 
    Moreover, for all $i \neq j$, the cusps of $\Sigma_{x_i}^\circ$ and $\Sigma_{x_j}^\circ$ are not contained in any common cusps of $M \sss L_\tau$.
\end{lemma}
\begin{proof}
First observe that for each each surface $\Sigma_x^\circ$, the inclusion into $M$ induces the trivial homomorphism  since it factors through the inclusion $\Sigma_x^\circ \to \Sigma_x$, and $\Sigma_x$ is either a  disk or is contained in a tetrahedron of $\tau$.

Suppose $\gamma$ is a loop in $M \sss L_\tau$ which is non-null-homotopic in $M$ (hence also in $M \sss L_\tau$).  From the previous paragraph, it follows that $\gamma$ cannot be homotoped to lie entirely inside any one of the surfaces $\Sigma_x^\circ$.  Since each component of $L_\tau$ is homotopically trivial in $M$, we see that $\gamma$ is non-peripheral (i.e.~not freely homotopic into a cusp) in $M \sss L_\tau$.  After a homotopy, we may therefore assume that $\gamma$ is geodesic with respect to the $\CAT(-1)$ metric, and we do so. 

We also assume that the isotopy class of $\Sigma_x^\circ$ is represented by a totally geodesic surface with respect to the $\CAT(-1)$ metric.  Therefore $\gamma$ intersects each $\Sigma_x^\circ$ transversely (possibly empty), and every intersection point is essential. 

Given a tetrahedron $T$, consider the family of surfaces $\mathcal S_\tau^\circ(T)$ consisting of surfaces $\Sigma_x^\circ$ for which $x$ is one of the following:
\begin{enumerate}
    \item A face of $T$;
    \item An edge adjacent to a vertex of $T$; or
    \item A tetrahedron $T'$ with $d_{\hat \tau}(v_T,v_{T'})\leq 2$.
\end{enumerate}

By inspection, the set of surfaces $\{ \Sigma_x \mid \Sigma_x^\circ \in \mathcal S_\tau^\circ(T)\}$
have the property that any connected component of \[M \sss \bigcup_{\Sigma_x^\circ \in S^{\circ}_{\tau}(T)} \Sigma_x\] that meets $T$ is contractible.  Indeed, each component meeting $T$ is either an open contractible subset of $T$, or is a neighorhood of a vertex $v \in T$, in the star of $v$.
Thus, if $\gamma$ intersects $T$, then it must have an essential intersection with some surface in $\mathcal S_\tau^\circ(T)$.

Now Lemma~\ref{lem:big combsys, many tetrahedra} implies that $\gamma$ intersects at least $n$ tetrahedra $T_0,\ldots,T_{n-1}$ for which the barycenters $v_{T_i}$ and $v_{T_j}$ are distance at least $8$ in $\hat \tau$ if $i \neq j$.  Let $\Sigma_{x_i}^\circ \in \mathcal S_\tau^\circ(T_i)$ be one of the surfaces essentially intersected by $\gamma$.  The cusps of $\Sigma_{x_i}^\circ$ are contained in cusps of $M \sss L_\tau$ that correspond to components of $L_\tau$ contained in the union of tetrahedra $T'$ with $d_{\hat \tau}(v_{T_i},v_{T'}) \leq 2$.  Consequently, the cusps of $M \sss L_\tau$ that contain the cusps of $\Sigma_i$ and $\Sigma_j$ are distinct if $i \neq j$.  This completes the proof.
\end{proof}

Recall that for a hyperbolic manifold $N$ and $\epsilon > 0$, the {\em $\epsilon$--thin part} of $N$ is the set $N_{(0, \epsilon)} = \{x \in N \mid \text{injrad}(x) < \epsilon\}$, and $N_{[\epsilon, \infty)} = \{x \in N \mid \text{injrad}(x) \geq \epsilon\}$ is the {\em $\epsilon$--thick part} of $N$. By the Margulis Lemma, there is an $\epsilon_0$ (depending only on the dimension of $N$) so that if $\epsilon \leq \epsilon_0$, then $N_{(0,\epsilon)}$ is a disjoint union of horoball cusp regions and collar neighborhoods of geodesics of length less than $2 \epsilon$.   See \cite{Thurston-Book} or \cite{benedetti} for more details.  If $N$ is a hyperbolic $3$--manifold and $\epsilon < \epsilon_0$, then the collars of closed geodesics in $N_{(0,\epsilon)}$ are called {\em Margulis tubes}.  For a hyperbolic surface, $\Sigma$, and $\epsilon \leq \epsilon_0$, the diameter of each component of the thick part $\Sigma_{[\epsilon, \infty)}$ is bounded above and below by constants that depend only on $\epsilon$ and $\chi(\Sigma)$. 

We will use the intuitive fact that when $\epsilon$ is small, curves intersecting $N_{[\epsilon, \infty)}$ that also enter Margulis tubes or cusps are long. This fact is made precise by the following theorem; see e.g.~\cite{bm, meyerhoff}.

\begin{theorem} \label{thm:thinpart} 
There exists a continuous, positive function $\mu \colon [0,\epsilon_0] \to \mathbb R$ such that $\displaystyle{\lim_{\epsilon \to 0} \mu(\epsilon) = \infty}$, with the following property.
Given $\epsilon \leq \epsilon_0$ and any hyperbolic manifold $N$, let $T_{\epsilon}$ be a connected component of $N_{(0, \epsilon)}$, and $T_{\epsilon_0} \supset T_\epsilon$ the component of the $\epsilon_0$--thin part containing it (these are either nested Margulis tubes or cusp neighborhoods).  Then
\[ d_{N}\big(\partial T_{\epsilon}, \partial T_{\epsilon_0}\big) \geq \mu(\epsilon).\]
\end{theorem}

We are now ready for the proof of the main theorem.

\medskip

\noindent
{\bf Theorem~\ref{thm:unbounded essential systole}}  {\em \UnbddEss}

\medskip

\begin{proof}
By Proposition~\ref{Prop:big CT exist}, there exists a sequence of Cooper-Thurston triangulations $\tau_n$ such that $\combsys(\hat \tau_n) \geq 16n$.  The theorem is a consequence of the following.

\medskip

\noindent
{\bf Claim.}
The essential systoles of $L_{\tau_n}$ tend to infinity; that is,
\[ \lim_{n \to \infty} {\sys}_e(L_{\tau_n}) = \infty .\]
\begin{proof}  Let $\gamma_n$ be closed geodesic in $M \sss L_{\tau_n}$ which is non-null-homotopic in $M$, and which realizes the essential systole of $L_{\tau_n}$. For each $n$ consider the surfaces 
\[ \Sigma_0^\circ(n),\ldots,\Sigma_{n-1}^\circ(n) \in \mathcal S_{\tau_n}^\circ,
\]
from Lemma~\ref{Lem:snaking through} that essentially intersect $\gamma_n$.  Since the surfaces in $\mathcal S_{\tau_n}^\circ$ are quasi-Fuchsian, we may homotope each $\Sigma_i^\circ(n)$ to a {\em pleated surface}; that is, the inclusion is homotopic to a $1$--Lipschitz map of a hyperbolic surface which is totally geodesic in the complementary regions of a geodesic lamination (see \cite[Theorem~8.10.8]{Thurstonnotes}).   The geodesic $\gamma_n$ intersects each $\Sigma_i^\circ(n)$ nontrivially in some point $z_i(n)$.

Now we assume $\ell(\gamma_n) = {\sys}_e(L_{\tau_n})$ does not tend to infinity with $n$ and derive a contradiction. 
This assumption implies that we may pass to a subsequence, and re-index so that for some $R >0$ we have $\ell(\gamma_n) < R$ for all $n$. 

The rest of the proof of the claim is divided into two cases.\\

\noindent
{\bf Case 1.} There is no $\epsilon > 0$ so that $z_i(n)$ is contained in the $\epsilon$--thick part of $\Sigma_i^\circ(n)$ for all $i$ and $n$.\\

Passing to a further subsequence if necessary, we can assume that there is some $z_{i_n}(n)$ in the $\frac{1}n$--thin part of $\Sigma_{i_n}^\circ(n)$ for all $n$ and some $i_n$.  Since a pleated surface is a $1$-Lipschitz map into $M \sss L_{\tau_n}$, it maps $\tfrac1n$--thin parts in each $\Sigma_i^\circ(n)$ to the $\tfrac1n$--thin parts in $M \sss L_{\tau_n}$.
In particular, $\gamma_n$ enters this $\tfrac1n$--thin part, and thus travels distance $\mu(\tfrac1n)$ into the $\epsilon_0$--thin part by \Cref{thm:thinpart}.  Choosing $n$ sufficiently large so that $\mu(\tfrac1n) > R$, it follows that $\gamma_n$ is entirely contained in a component $T_{\epsilon_0}$ of the $\epsilon_0$--thin part of $M \sss L_{\tau_n}$.

First suppose $T_{\epsilon_0}$ is a cusp.  This cusp is a tubular neighborhood of a component of $L_{\tau_n}$, minus the component.  Every component of $L_{\tau_n}$ bounds a disk in $M$, and in particular the inclusion of $T_{\epsilon_0}$ into $M$ is null-homotopic.  Hence so is $\gamma_n$, a contradiction.

Next, we assume $T_{\epsilon_0}$ is a Margulis tube.  Since $z_{i_n}(n)$ is in the $\tfrac1n$--thin part of $\Sigma_{i_n}^\circ(n)$, the core of $T_{\epsilon_0}$ must be a geodesic with a power that maps into $\Sigma_{i_n}^\circ(n)$.  Consequently, in the locally CAT(-1) metric on $M \sss L_{\tau_n}$, the  geodesic representative of this power of $\gamma_n$ must lie in the totally geodesic representative of $\Sigma_{i_n}^\circ(n)$.  Since $\Sigma_{i_n}^\circ(n)$ is an embedded totally geodesic surface, $\gamma_n$ must be contained in it.  But $\Sigma_{i_n}^\circ(n)$ is null-homotopic in $M$; hence, so is $\gamma_n$, another contradiction.\\

\noindent
{\bf Case 2.}  There exists $0 < \epsilon < \epsilon_0$ so that $z_i(n)$ is contained in the $\epsilon$--thick part of $\Sigma_i^\circ(n_i)$ for all $i$ and $n$.\\

Without loss of generality, we assume that $\epsilon$ is small enough so that distinct $\epsilon$--thin parts of $M \sss L_{\tau_n}$ are $1$--separated. Each connected component of the $\epsilon$--thick part of $\Sigma_i^\circ(n)$ has uniformly bounded diameter since $\Sigma_i^\circ(n)$ has bounded Euler characteristic (it is a sphere with at most $11$ punctures).  Thus there is a boundary component of the $\epsilon$--thick part of $\Sigma_i^\circ(n)$ within some fixed distance $\delta$ from $z_i(n)$ for all $i$ and $n$.  For each $i$ and $n$, let $w_i(n)$ be a point in the $\epsilon$--thin part of $\Sigma_i^\circ(n)$ which is distance at most $\delta+1$ from $z_i(n)$.

For each $n$, we note that the points $w_0(n),\ldots,w_{n-1}(n)$ are in distinct thin parts of $M \sss L_{\tau_n}$.  This follows directly from the lemma if the thin parts are horoball cusps, and otherwise it follows by considering the totally geodesic representatives of the isotopy classes of the surfaces in the $\CAT(-1)$ metric: since the surfaces are pairwise disjoint, no geodesic in one is homotopic to a curve in another. 

Since $\{w_0(n),\ldots,w_{n-1}(n)\}$ are in distinct $\epsilon$--thin parts, which are $1$--separated, this set of points is also $1$--separated.
Lift $\gamma_n$ to a geodesic path $\widetilde \gamma_n$ of length at most $R$ in the universal cover $\mathbb H^3$ based at some point $p$.  For each $z_i(n) \in \gamma_n$, let $\widetilde z_i(n)$ be a point on $\widetilde \gamma_n$ that projects to $z_i(n)$ and $\widetilde w_i(n)$ be a point within distance $\delta+1$ of $\widetilde z_i(n)$ that projects to $w_i(n)$. See \Cref{fig:thickthin}.
Then $\{\widetilde w_0(n),\ldots,\widetilde w_{n-1}(n)\}$ is also a $1$--separated set of $n$ points.
On the other hand, for each $n$, these points are contained in a ball of fixed radius $R+\delta+1$ in $\mathbb H^3$.  This is a contradiction for $n$ sufficiently large since a $1$--separated set in such a ball contains at most $\frac{V(R+\delta+3/2)}{V(1/2)}$ points, where $V(r)$ is the volume of a hyperbolic ball of radius $r$.

\textbf{\begin{figure}[H]
    \centering
    \begin{tikzpicture}
    \node at (0,0) {\includegraphics[width=0.6\linewidth]{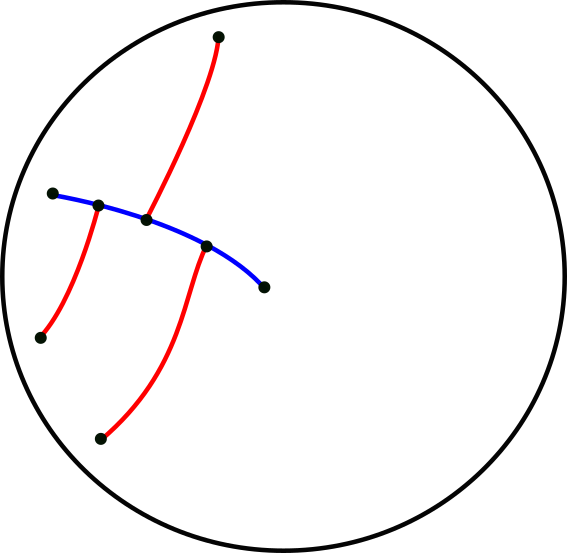}};
    \node at (-2.65,-.9) {$\widetilde{w}_i(n)$};
    \node at (-1.9,-2.4) {$\widetilde{w}_k(n)$};
    \node at (-1.8,.4) {$\widetilde{z}_j(n)$};
    \node at (-.3,3) {$\widetilde{w}_j(n)$};
    \node at (-2.3,1.25) {$\widetilde{z}_i(n)$};
    \node at (-.8,.75) {$\widetilde{z}_k(n)$};
    \node at (0,0) {$p$};
    \node at (-0.75,-0.1) {\color{blue}$\widetilde{\gamma_n}$};
    \node at (1.5,-2) {$B(p, R + \delta + 1)$};
    \end{tikzpicture}
    \caption{The points $\{\widetilde{w}_0(n), \widetilde{w}_1(n),\ldots, \widetilde{w}_{n - 1}(n)\}$ are all contained in the ball $B(p, R + \delta + 1) \subset \mathbb{H}^3$.}
    \label{fig:thickthin}
\end{figure}}

Since a uniform bound on ${\sys}_e(L_{\tau_n})$ for any subsequence produces a contradiction, it follows that
\[ \lim_{n \to \infty} {\sys}_e(L_{\tau_n}) = \infty,\]
as required.
\end{proof}

As already noted, the claim implies the theorem.
\end{proof}

\section{Concluding Remarks}\label{sec:conclusion}

As indicated in the introduction, if one is only interested in proving the links $L_\tau$ can be taken to be filling, we can shorten the proof a little, and avoid using Thurston's Hyperbolization Theorem.  We now sketch how that can be done.
The locally $\CAT(-1)$ metric on $M \sss L_\tau$ has universal cover that is also $1$--hyperbolic (since triangles are thinner than their comparison triangles in $\mathbb H^2$ which are $1$--slim).  Consequently, we could apply a Theorem~\ref{thm:KW} to this metric, and deduce the same conclusion in Theorem~\ref{thm:big essential systole filling} (that is, sufficiently large locally $\CAT(-1)$ essential systole implies filling).  In this metric, the complement  $T \sss (L_\tau \cap T)$ in any tetrahedron of $T$ of $\tau$ has {\em exactly} the hyperbolic structure constructed in \S\ref{Sec:Tet tangle}.  Consequently, we can find $\varepsilon > 0$ so that the totally geodesic representatives of disjoint surfaces $\Sigma_x^\circ,\Sigma_y^\circ \in \mathcal S_\tau^\circ$ which do not share a cusp in $M \sss L_\tau$ have disjoint $\varepsilon$--neighborhoods.  If $\varepsilon$ is small enough, the $\varepsilon$--thin parts are precisely horoball cusp neighborhoods and are $\varepsilon$--separated.  Now if $\gamma$ is a geodesic in this metric realizing the locally $\CAT(-1)$ essential systole for $L_\tau$, and $n$ is as in Lemma~\ref{Lem:snaking through}, then $\gamma$ intersects at least $n$ pairwise disjoint surfaces $\Sigma_0^\circ,\ldots,\Sigma_{n-1}^\circ$, no two of which share a cusp.  Consequently, the length  $\gamma$ is at least $n\varepsilon$, which can be made arbitrarily large.

\subsection{Questions}
We now list some questions for further investigation, proposed by Ian Biringer after originally circulating our preprint.
The first asks for strengthening of Theorem~\ref{thm:main}.
\begin{question}[Biringer]
Does every closed, orientable $3$--manifold, $M$, contain a filling knot? Given $R>0$, does $M$ contain a knot with essential systole at least $R$?
\end{question}
\begin{remark}
    The authors have now resolved this question affirmatively, but with a minor (natural) variation on the definition of essential systole, in which one considers lengths of geodesics (as opposed to arbitrary closed loops).  See \cite{leiningerwu2}.
\end{remark}
Theorem~\ref{thm:fullrank2} inspires the following.
\begin{question}[Biringer] Given a closed, orientable $3$--manifold, $M$, is there a link $L$ that is full rank--$n$ filling, for all $n >0$?
\end{question}
The answer is vacuously yes for some special $3$--manifolds, e.g. those for which there is a uniform bound on the rank of a subgroup of the fundamental group (e.g.~the $3$--torus and spherical $3$--manifolds).

\bibliographystyle{alpha}
\bibliography{main}

\end{document}